\documentclass[10pt, a4paper, final]{amsart}
\usepackage{amstext, amscd, amsthm, bm, amssymb, dsfont}
\usepackage{amsmath}
\usepackage[english]{babel}
\usepackage{booktabs}
\usepackage{csquotes}
\usepackage[style=ieee-alphabetic,url=false,isbn=false,giveninits=true,backend=biber,maxnames=5,maxalphanames=5]{biblatex}
\bibliography{../bib}

\usepackage{color}
\usepackage{datetime}
\usepackage{enumitem}
\usepackage{graphicx}
\usepackage{wasysym}
\usepackage{mathtools}
\usepackage[scr=rsfso]{mathalfa}
\usepackage{microtype}
\usepackage{textcomp}
\usepackage{hyperref}
\DeclareSymbolFont{extraup}{U}{zavm}{m}{n}
\DeclareMathSymbol{\varheart}{\mathalpha}{extraup}{86}

\usepackage{savesym}
\usepackage{stmaryrd}
\usepackage{textpos}
\savesymbol{iint}
\restoresymbol{TXF}{iint}
\usepackage{tikz}
\usetikzlibrary{babel}
\usepackage{tikz-cd}

\usepackage[backgroundcolor=blue!40, linecolor=lightgray, bordercolor=lightgray, textsize=tiny]{todonotes}

\usepackage[notref]{showkeys}

\newcommand{\R}             {\mathbb R} 
    
\newcommand{\N}             {\mathbb N}
\newcommand{\Z}             {\mathbb Z}
\newcommand{\Q}             {\mathbb Q}

\newcommand{\cat}[1]{\pmb{\mathrm{#1}}}

\usepackage{xparse}
\renewcommand{\mathbf}[1]{\ensuremath{\pmb{\mathrm{#1}}}}

\DeclareDocumentCommand{\Ddeath}{ g }{
  \ensuremath{
    \mathrm{D}^{\dagger\IfNoValueF{#1}{,#1}}
  }
}

\DeclareDocumentCommand{\Bdeath}{ g }{
  \ensuremath{
    \mathbf{B}_{\IfNoValueF{#1}{#1}}^{\dagger}
  }
}
\DeclareDocumentCommand{\Brigdeath}{ m g }{
  \ensuremath{
    \mathbf{B}_{\mathrm{rig}, #1}^{\dagger\IfNoValueF{#2}{,#2}}
  }
}

\DeclareDocumentCommand{\Drigdeath}{ g }{
  \ensuremath{
    \mathrm{D}_{\mathrm{rig}}^{\dagger\IfNoValueF{#1}{,#1}}
  }
}
\DeclareDocumentCommand{\Ars}{ g g }{
  \ensuremath{
    \mathbf{A}_{\IfNoValueF{#1}{#1}}^{\IfNoValueF{#2}{#2}}
  }
}

\DeclareDocumentCommand{\Brs}{ g g }{
  \ensuremath{
    \mathbf{B}_{\IfNoValueF{#1}{#1}}^{\IfNoValueF{#2}{#2}}
  }
}

\DeclareDocumentCommand{\Drs}{ g }{
  \ensuremath{
    \mathrm{D}^{\IfNoValueF{#1}{#1}}
  }
}

\usetikzlibrary{arrows,calc,matrix,decorations.pathmorphing}
\tikzset{sra/.code=
  {\pgfkeysalso{
      /tikz/decorate,
      /tikz/decoration={snake,amplitude=1pt,
        segment length=5pt, pre=lineto,
        pre length=4pt, post=lineto, post length=8pt}
    }}
}

\newcommand*{\ra}[1][]{\raisebox{-1pt}{\begin{tikzpicture}%
        \draw[xscale=1,thin,shorten >=3pt, >=stealth, ->,font=\scriptsize] (0,0)%
                 node{\hspace*{-2pt}} -- (0.5,0) node[above] {#1}%
                  -- node{} (1,0);\end{tikzpicture}}\penalty1000\relax}
\newcommand*{\sra}[1][]{\raisebox{-1pt}{\begin{tikzpicture}%
      \draw[->, 
      xscale=1,thin,shorten >=3pt, >=stealth, font=\scriptsize, sra](0,0)%
                 node{\hspace*{-2pt}} -- (0.5,0) node[above] {#1}%
                  -- node{} (1,0);\end{tikzpicture}}\penalty1000\relax}

\renewcommand*{\mapsto}{\raisebox{-1pt}{\begin{tikzpicture}%
       \draw[xscale=1,thin,shorten >=3pt, >=stealth, |->] (0,0)%
                node{\hspace*{0pt}}%
                -- node{} (1,0);\end{tikzpicture}}\penalty1000\relax}

\newcommand*{\inj}[1][]{\raisebox{-1pt}{\begin{tikzpicture}%
        \draw[xscale=1,thin,shorten >=3pt, >=stealth, right hook->,%
                 font=\scriptsize] (0,0)%
                 node{\hspace*{0pt}} -- (0.5,0) node[above] {#1} --%
                 node{} (1,0);\end{tikzpicture}}\penalty1000\relax}

\newcommand*{\surj}[1][]{\raisebox{-1pt}{\begin{tikzpicture}%
       \draw[xscale=1,thin,shorten >=3pt, >=stealth, ->>,font=\scriptsize] (0,0)%
                node{\hspace*{0pt}} -- (0.5,0) node[above] {#1}%
                -- node{} (1,0);\end{tikzpicture}}\penalty1000\relax}

\newcommand*{\isom}{\raisebox{-1pt}{\begin{tikzpicture}%
       \draw[xscale=1,thin,shorten >=3pt, >=stealth, ->] %
                (0,0) node{\hspace*{0pt}} -- node{} (1,0);%
                \draw[xscale=1] (0.4,0.1) node {$\sim$};\end{tikzpicture}}%
        \penalty1000\relax}

\newcommand*{\lra}[1][]{\raisebox{-1pt}{\begin{tikzpicture}%
       \draw[xscale=1,thin,shorten >=3pt, >=stealth, <->,font=\scriptsize] (0,0) %
                node{\hspace*{-2pt}} -- (0.5,0) node[above] {#1} --%
                node{} (1,0);\end{tikzpicture}}\penalty1000\relax}

\makeatletter
\newcommand{\osetwv}[2]{%no math
  {\mathop{#2}\limits^{\vbox to -4\ex@{\kern-\tw@\ex@
   \hbox{\hspace{-0.2em}\scriptsize #1}\vss}}}}
\makeatother

\makeatletter
\newcommand{\oset}[2]{%math
  {\mathop{#2}\limits^{\vbox to -5\ex@{\kern-\tw@\ex@
   \hbox{\hspace{-0.25em}\scriptsize$#1$}\vss}}}}
\makeatother

\tikzset{commutative diagrams/.cd,
                     arrow style=tikz,
                     diagrams={>=stealth', line width=0.5pt}}

\DeclareRobustCommand{\gobblefive}[5]{}

\DeclarePairedDelimiter\abs{\lvert}{\rvert}%
\DeclarePairedDelimiter\norm{\lVert}{\rVert}%
\DeclarePairedDelimiter{\ceil}{\lceil}{\rceil}

\makeatletter
\let\oldabs\abs
\def\abs{\@ifstar{\oldabs}{\oldabs*}}
\let\oldnorm\norm
\def\norm{\@ifstar{\oldnorm}{\oldnorm*}}
\makeatother

\DeclareMathOperator{\Hom}{Hom}

\DeclareMathOperator{\Ext}{Ext}

\DeclareMathOperator{\Ind}{Ind}

\DeclareMathOperator{\image}{im}

\DeclareMathOperator{\colim}{colim}

\DeclareMathOperator{\id}{id}

\DeclareMathOperator{\ext}{ext}
\DeclareMathOperator{\sgn}{sgn}
\DeclareMathOperator{\shuffle}{shuffle}
\DeclareMathOperator{\total}{tot}

\newcommand{\forget }{\cat{\text{\textquestiondown}}}
\newcommand{\mbof}[2]{\overset{\scriptscriptstyle #2}{#1}}

\usepackage{cleveref}

\crefname{none}{}{}
\crefname{(none)}{}{}
\creflabelformat{(none)}{(#2#1#3)}
\crefname{diagram}{diagram}{diagrams}
\creflabelformat{diagram}{(#2#1#3)}

\newtheoremstyle{mythm}%                % Name
  {}%                                     % Space above
  {}%                                     % Space below
  {\itshape}%                             % Body font
  {}%                                     % Indent amount
  {\bfseries\scshape}%                            % Theorem head font
  {.}%                                    % Punctuation after theorem head
  { }%                                    % Space after theorem head, ' ', or \newline
  {}%                                     % Theorem head spec (can be left empty, meaning `normal')
  
\newtheoremstyle{myprop}%                % Name
  {}%                                     % Space above
  {}%                                     % Space below
  {\itshape}%                             % Body font
  {}%                                     % Indent amount
  {\bfseries}%                            % Theorem head font
  {.}%                                    % Punctuation after theorem head
  { }%                                    % Space after theorem head, ' ', or \newline
  {}%                                     % Theorem head spec (can be left empty, meaning `normal')

\newtheoremstyle{myrmk}%                % Name
  {}%                                     % Space above
  {}%                                     % Space below
  {}%                             % Body font
  {}%                                     % Indent amount
  {\bfseries}%                            % Theorem head font
  {.}%                                    % Punctuation after theorem head
  { }%                                    % Space after theorem head, ' ', or \newline
  {}%                                     % Theorem head spec (can be left empty, meaning `normal')

\theoremstyle{mythm}
\newtheorem{theorem}{Theorem}[section]
\newtheorem*{theorem*}{Theorem}

\theoremstyle{myprop}
\newtheorem{lemma}[theorem]{Lemma}
\newtheorem{corollary}[theorem]{Corollary}

\newtheorem{proposition}[theorem]{Proposition}
\newtheorem{prototheorem}[theorem]{Prototheorem}

\theoremstyle{myrmk}
\newtheorem{remark}[theorem]{Remark}
\newtheorem{example}[theorem]{Example}

\newtheorem{definition}[theorem]{Definition}
\newcommand{\rom}[1]{\uppercase\expandafter{\romannumeral #1\relax}}

\makeindex

\author{Oliver Thomas}
\address{Mathematisches Institut, Im Neuenheimer Feld 205, D-69120 Heidelberg}
\email{othomas@mathi.uni-heidelberg.de}
\title{Cohomology of Topologised Monoids}

\begin{document}
\renewcommand{\ref}[1]{\cref{#1}}
\begin{abstract}
  We prove standard results of group cohomology -- namely, existence of a long
  exact sequence, classification of torsors via the first cohomology group,
  Shapiro's lemma, the Hochschild-Serre spectral sequence, a decomposition of
  the cochain complex in the direct product case, and Jannsen's result on the
  recovery problem -- for cohomology theories such as continuous, analytic,
  bounded, and pro-analytic cohomology. We also prove these results for certain
  monoids.

  The cohomology groups considered here all have very concrete interpretations
  by means of a cochain complex. Therefore, we do not use methods of homological
  algebra, but explicit calculations on the level of cochains, using techniques
  dating back to Hochschild and Serre.
\end{abstract}
\maketitle
\tableofcontents
\section*{Introduction}

There are numerous variants of group cohomology defined via cochains: They can
be assumed to be continuous, analytic, bounded etc. In these cases, they however
lose their functorial properties: The obvious coefficient categories rarely
admit quotient objects. The most elegant way of fixing this issue is to admit a
larger category of coefficients by taking a sheaf-theoretic point of view
(cf.~\cite{MR2422342}). For many applications though, this point of view simply
shifts the issue: Considering only cohomology groups defined via cochains, it is
unclear whether something like a Hochschild-Serre spectral sequence exists. If
we admit a larger category, the existence of such a spectral sequence is clear,
but it is unclear what the objects appearing in said spectral sequence look
like.

Sometimes, these objects are however of importance. In this article we try to
push the cochain point of view as far as we can. On the one hand, this means that
the results we will prove hold for all aforementioned topologised group
cohomology theories. On the other hand, this means that we generally follow the
\emph{direct method} of Hochschild and Serre without appealing to homological
arguments. In contrast to arguments from homological algebra, ours will be
ad-hoc, combinatorial in nature and very calculation intensive.

While the direct method is ridiculously flexible, it is quite hard to present
one streamlined proof that shows off many of its features. For this reason, we
will introduce an axiomatic framework that allows us to deal with all of
these variants in one go.

One application we have in mind is analytic cohomology of
$(\varphi,\Gamma)$-modules in the sense of \cite{MR3444228}. For this reason, we
are strictly speaking not setting up a framework for topologised groups, but for
topologised monoids. This presents additional difficulties. 

\subsection*{Organisation}

We start by introducing a framework of \emph{topologised categories} and how to
define group (and monoid) actions for them in \ref{sec:topological-categories,sec:topologised-groups-and-monoids,sec:rigidified-g-modules,sec:induced-module}. While the examples we have in mind
are very concrete, the framework itself is rather abstract -- especially for the
later topic of rigidifications. We encourage the reader to mainly think about
the examples listed in \ref{ex-tbl:rigidifications}. After a brief discussion of
abstract monoid cohomology in \ref{sec:abstract-monoid-cohomology}, we fix the setup for
the following sections in \ref{sec:setup-topological-cochains}. Afterwards we
show how both the existence of a long exact sequence and the classification of
torsors hold in our setting
(cf.~\ref{sec:cohomology-of-topologised-monoids,sec:cohomology-and-extensions}).
The first glimpse at how much we should appreciate arguments of homological
algebra is then given in \ref{sec:shapiro-top-grps}, where we prove Shapiro's
lemma for topologised groups. (Shapiro's lemma for topologised monoids will have
to wait until \ref{sec:shapiro-for-monoids}.) One of the main motivations to
introduce the framework however was the spectral sequence of Hochschild and
Serre. In \ref{sec:hochschild-serre} we follow the arguments of Hochschild's and 
Serre's original article to prove it in our setting. As for applications not
just the existence of the spectral sequence, but also the existence of a
quasi-isomorphism of complexes is of interest, we also prove a corresponding
result in \ref{sec:double-complex}.

\subsection*{Acknowledgements}

This article is based upon parts of its author's PhD thesis
(cf.~\cite{thomas:on-analytic-and-iwasawa-cohomology}). We express our sincerest
gratitude to Otmar Venjakob for many fruitful discussions.

\section{Topological Categories}
\label{sec:topological-categories}
There are a number of notions of topological categories in the literature, none
of which is standard. For our purposes it is sufficient to have a good notion of
discrete spaces.
\begin{definition}
  A \emph{concrete category} is a faithful functor
  $\forget\colon\cat{C}\ra\cat{Set}$. One often only says that a category
  $\cat{C}$ is a concrete category, even though the forgetful functor $\forget$
  is an essential part of the datum.
\end{definition}
\begin{definition}
  A concrete category $\forget \colon\cat{C}\ra\cat{Set}$ is called a
  \emph{category admitting discrete objects}, if $\cat{C}$ has finite
  limits and $\forget$ admits a fully faithful left adjoint $\cat{F}.$
 
  We will denote by \[\cat{C}^\delta = \{X\in \cat{C}\mid
    \forget(\Hom_{\cat{C}}(X,-))=\Hom_{\cat{Set}}(\forget X,\forget (-))\}\] the
  \emph{discrete objects} in $\cat{C}$. This terminology is justified as all
  objects in $\cat{Set}$ give rise to discrete objects,
  cf.~\ref{prop:discrete-objs-are-sets}. By abuse of notation, we will often
  only say that a category admits discrete objects without specifying the
  forgetful functor.
\end{definition}

We will always denote by $\bullet$ a singleton set.

\begin{remark}
  Note that for a category admitting discrete objects, the forgetful functor is
  represented by $\cat{F}\bullet$.
\end{remark}
\begin{proposition}
  \label{prop:cat-w-discrete-objs-has-right-inverse-to-forgetting}
  Let $\cat{C}$ be a category admitting discrete objects. Then
  $\forget\circ\cat{F}\simeq \id_{\cat{Set}}.$
  \begin{proof}
    The isomorphism $\Hom_{\cat{Set}}(X,Y)\isom \Hom_{\cat{C}}(\cat{F}X,\cat{F}Y) =
    \Hom_{\cat{Set}}(X,\forget \cat{F}Y)$ shows that
    $Y\cong\forget\cat{F}Y$.
  \end{proof}
\end{proposition}
\begin{proposition}
  \label{prop:cat-w-discrete-objs-forgets-monos-to-monos}
 Let $\cat{C}$ be a category admitting discrete objects. Then
 $\forget $ maps monomorphisms to monomorphisms.
 \begin{proof}
   Let $X\ra Y$ be a monomorphism in $\cat{C}.$ It suffices to show that for
   $\alpha\neq\beta\in\Hom_{\cat{Set}}(\bullet, \forget X)$ also
   their induced maps in $\Hom_{\cat{Set}}(\bullet, \forget Y)$
   differ. Assume they did not.
   $\alpha$ and $\beta$ correspond to $\alpha',
   \beta'\in\Hom_{\cat{C}}(\cat{F}\bullet, X)$. As $\forget\cat{
     F}\simeq \id_{\cat{Set}}$ by \cref{prop:cat-w-discrete-objs-has-right-inverse-to-forgetting}, this implies
   that \[\forget \left( \cat{F}\bullet \oset{\alpha'}{\ra} X \ra
     Y\right) = \forget \left( \cat{F}\bullet\oset{\beta'}{\ra} X
   \ra Y\right),\] so $\alpha'=\beta'$ as $X\ra Y$ was assumed to be mono.
 \end{proof}
\end{proposition}

\begin{proposition}
  \label{prop:subobjs-of-discretes-are-discrete}
  Subobjects of discrete objects are discrete.
  \begin{proof}
    Let $X$ be a discrete object in a category admitting discrete objects $\cat{C}$,
    $D$ a subobject and $Y$ an arbitrary object in $\cat{C}$. We have the
    following commutative diagram:
    \[\begin{tikzcd}
        \Hom_{\cat{C}}(X,Y) \arrow{r}{\cong}\arrow{d} &
        \Hom_{\cat{Set}}(\forget  X, \forget  Y)\arrow{d}\\
        \Hom_{\cat{C}}(D,Y) \arrow[hook]{r}{\forget } &
        \Hom_{\cat{Set}}(\forget  D, \forget  Y)
      \end{tikzcd}\]
   As $\forget D\ra\forget X$ is again a mono by \cref{prop:cat-w-discrete-objs-forgets-monos-to-monos}, the
   map on the right is surjective and hence so is the bottom one, i.\,e., $D$ is discrete. 
  \end{proof}
\end{proposition}

\begin{proposition}
  \label{prop:discrete-objs-are-sets}
  Let $\cat{C}$ be a category admitting discrete objects. Then $\cat{F}S$ is
  discrete for every set $S$. $\cat{F}$ is essentially
  surjective onto the discrete objects, so
  $\cat{F}\colon\cat{Set}\ra\cat{C}^\delta$ is an equivalence of categories.
  \begin{proof}
    Let $S$ be a set. Then
    \[\Hom_{\cat{C}}(\cat{F}S, Y) \oset{\forget }{\ra} \Hom_{\cat{Set}}(\forget\cat{
      F}S, \forget  Y) = \Hom_{\cat{C}}(\cat{F}\forget\cat{
      F}S, Y)=\Hom_{\cat{C}}(\cat{F}S,Y),\] and the identifications imply that
    the first map is an isomorphism.

    On the other hand, if $X$ is discrete, i.\,e., if
    \[\Hom_{\cat{C}}(X,Y)=\Hom_{\cat{Set}}(\forget X,\forget 
    Y) = \Hom_{\cat{C}}(\cat{F}\forget X, Y) \text{ for all } Y,\] then
    Yoneda implies $X\cong\cat{F}\forget X,$ so $\cat{F}$ is
    essentially surjective onto discrete objects.
  \end{proof}
\end{proposition}

\begin{example}
  Categories admitting discrete objects don't quite behave like topological
  spaces, as for example singleton objects need not be isomorphic. Consider the following category $\cat{C}$ where the objects are tuples
  $(A,\tau_A)$ with $A$ any set and $\tau_A$ \emph{any} subset of the power set $2^A$
  of $A$. A morphism $(A,\tau_A)\ra(B,\tau_B)$ in $\cat{C}$ is defined as a map
  $f\colon A\ra B$ subject to the condition that for $b\in\tau_B$,
  $f^{-1}(b)\in\tau_A.$
  
  The forgetful functor has the obvious left adjoint $A\mapsto (A,2^A)$, but
  singleton sets need not be isomorphic:
  $(\bullet,\varnothing)\not\cong(\bullet, 2^\bullet)$ and only the latter
  object is discrete while only $(\bullet,\varnothing)$ is final in $\cat{C}.$
\end{example}

\begin{definition}
  We say that a category $\cat{C}$ admitting discrete objects is \emph{topological}, if
  the functor $\cat{F}\colon \cat{S}\ra\cat{C}$ commutes with finite limits and
  if for every discrete object $D$ and all objects $X,Y$ the natural map \[
    \Hom_{\cat{C}}(D\times X,Y) \ra \Hom_{\cat{Set}}(\forget D,
    \Hom_{\cat{C}}(X,Y))\] is a bijection.
\end{definition}
\begin{remark}
  If $\cat{F}$ commutes with finite limits, it especially maps a final object to
  a final object.
\end{remark}
\begin{remark}
  \label{rmk:internal-hom-in-cgwh-and-discrete-objects}
  The isomorphism
\[
    \Hom_{\cat{C}}(D\times X,Y) \ra \Hom_{\cat{Set}}(\forget D,
    \Hom_{\cat{C}}(X,Y))\] replaces some kind of internal Hom-functor:
  In the category of compactly generated weakly Hausdorff spaces admits we endow
  for spaces $X,Y$ the set $\Hom_{\cat{CGWH}}(X,Y)$ with the compact open
  topology and call the resulting object $[X,Y]$. This results in a pair of
  adjoint functors: \[\Hom_{\cat{CGWH}}(Z\times X,Y)\cong
    \Hom_{\cat{CGWH}}(Z,[X,Y]).\] If $Z$ is furthermore discrete, this reads as
\[\Hom_{\cat{CGWH}}(Z\times X,Y)\cong
    \Hom_{\cat{CGWH}}(Z,[X,Y]) = \Hom_{\cat{Set}}(\forget Z,
    \Hom_{\cat{CGWH}}(X,Y)),\] which is precisely the second requirement we
  posed for a category admitting discrete objects to be topological. 
\end{remark}

\begin{example}
  Examples of topological categories are: the category of topological spaces, of
  Hausdorff topological spaces, of metric spaces -- all with continuous maps.
  The category of analytic manifolds (over some arbitrary base) is also
  topological. In all cases, $\forget$ is the obvious forgetful functor to
  $\cat{Set}$ and $\cat{F}$ maps a set to the same set with the discrete
  topology. Here we regard discrete sets as zero-dimensional manifolds.
\end{example}
% \begin{remark}
%   By \ref{prop:discrete-objs-are-sets}, the last condition is automatical if
%   $\cat{S}=\cat{Set}.$ For a topological category, we require that the category
%   $\cat{S}$ is chosen maximally. Metric topological spaces with bounded maps and
%   compact topological spaces are
%   topological categories. In both cases, $\cat{S}$ is the category of finite
%   sets.
% \end{remark}

\begin{proposition}
  \label{prop:top-cat-has-const-morph}
  In a topological category, every constant map of sets lifts to a morphism.
  \begin{proof}
    A constant map of sets factors as \[
      \forget X\ra\bullet\ra\forget  Y.\] As
    $\cat{F}\bullet$ is terminal in a topological category, this factorisation
    lifts to morphisms in the topological category.
  \end{proof}
\end{proposition}

\section{Topologised Groups and Monoids}
\label{sec:topologised-groups-and-monoids}
\begin{definition}
  A \emph{topologised group} is a group object in a topological category.
  Similarly, a \emph{topologised monoid} will mean a monoid object in a
  topological category. A morphism of topologised groups is a morphism
  $\phi\colon G\ra H$ in the ambient category such that the
  diagram \[\begin{tikzcd}
      G\times G\arrow{r}{\text{mul}}\arrow{d}{(\phi,\phi)} & G\arrow{d}{\phi}\\
      H\times H\arrow{r}{\text{mul}} & H
    \end{tikzcd}\]
  commutes. For a morphism of topologised monoids we furthermore require the
  commutativity of the following diagram: \[\begin{tikzcd} G\arrow{rr}{\phi} &
      & H\\ & 1 \arrow{ur}\arrow{ul} & \end{tikzcd}.\] Here $1$ is the trivial
  group structure on a final object of the ambient category and the morphisms
  $1\ra H, 1\ra G$ are the inclusion of identity elements.
\end{definition}

\begin{remark}
  Note that if $\cat{C}$ is a topological category, $\forget $
  maps topologised groups to groups and $\cat{F}$ maps groups to (discrete) group
  objects in $\cat{C}$. We again have an equivalence between the category of
  (abstract) groups and discrete topologised groups via $\cat{F}$.
\end{remark}

\begin{remark}
  If this topological category is the category of (Hausdorff) topological
  spaces, our notions of topologised groups and monoids coincide with the
  standard ones of topological groups and monoids. Other important examples are
  the categories of $L$-analytic manifolds where $L$ is a local field.
\end{remark}

\begin{definition}
  A morphism $N\ra G$ of topologised groups is called \emph{normal}, if its
  cokernel exists in the category of topologised groups with kernel exactly
  $N\ra G$. The cokernel $G\ra C$ will also be called the \emph{quotient} of $G$
  by $N$ and we will simply write $C\cong G/N$. A morphism $U\ra G$ is called
  an open normal subgroup, if it is normal and $G/U$ is discrete.
\end{definition}
\begin{remark}
  Note that this definition allows us to avoid the notion of \emph{strictness},
  which is rather cumbersome, cf.~\cref{rmk:image-coimage-are-horrible-defs}.
  Indeed, consider the bijective morphism $\R^\delta \ra \R$ in the category of
  locally compact groups, where $\R^\delta$ carries the discrete topology and
  $\R$ the usual one. It is easy to see that the cokernel of this morphism is
  the trivial morphism $\R\ra 1$, which has kernel $\R\oset{\id}{\ra}\R$, so
  $\R^\delta\ra\R$ is not normal.
\end{remark}
\begin{proposition}\label{prop:discrete-quotient-is-set-quotient}
  Let $U\ra G$ be an open normal subgroup of topologised groups. Then \[
    \forget (G/U) \cong \forget G /
    \forget  U.\]
    \begin{proof}
      Note that $\forget $ commutes with arbitrary limits and
      especially with taking kernels. Therefore
      \begin{align*}
        \Hom_{\cat{Grp}}(\forget (G/U), H) & = \forget  \Hom_{\cat{Grp_C}}(G/U,\cat{F}H)\\
                                                  &= \forget 
                                                    \left\{f\in \Hom_{\cat{Grp_C}}(G,\cat{F}H)\mid \ker f \supseteq U \right\} \\
                                                  &\subseteq \left\{ f \in \Hom_{\cat{Grp}}(\forget G,H) \mid \ker f \supseteq \forget  U \right\} \\
                                                &=
                                                  \Hom_{\cat{Grp}}(\forget 
                                                  G/\forget U,
                                                  H),
      \end{align*}
      which yields a surjection \[\forget G/\forget U \surj\forget (G/U),\] as
      epimorphisms in the category of groups are exactly the surjective group
      homomorphisms. On the other hand we also have a natural injection $\forget
      G/\forget U \inj\forget (G/U)$ as $\forget $ preserves kernels. As both
      maps clearly coincide, this proves the proposition.
    \end{proof}
\end{proposition}

\begin{proposition}\label{prop:discrete-quotient-has-section}
  Let $G$ be a topologised group with open normal subgroup $U$. Then $G\ra G/U$
  admits a section in $\cat{C}.$
  \begin{proof}
    Consider any section $\forget G/\forget U \ra
    \forget  G$, which by
    \cref{prop:discrete-quotient-is-set-quotient} is a section
    $\forget (G/U)\ra\forget G$. As $G/U$ is
    discrete, this lifts to a section in $\cat{C}.$
  \end{proof}
\end{proposition}
\begin{proposition}\label{prop:section-implies-quotient-is-set-quotient}
  Let $G$ be a topologised group with normal subgroup $N.$ If $G\ra G/N$ admits
  a section in $\cat{C}$, then $\forget(G/N)\cong\forget G/\forget N.$
  \begin{proof}
    As $\forget$ preserves kernel, we always have an injection $\forget
    G/\forget N\ra \forget(G/N).$ The existence of a section implies that it is
    also surjective and hence an isomorphism of groups.
  \end{proof}
\end{proposition}
\begin{proposition}\label{prop:quotient-well-behaved-after-product-with-monoid}
  Let $G'$ be a topologised group, $M$ a discrete topologised monoid and $U$ an
  open normal subgroup of $G'$. Then $U\times 1\ra G'\times M$ is the kernel of
  $G'\times M\ra G'/U\times M$ and the latter map is the cokernel of the former
  map in the category of topologised monoids.
  \begin{proof}
    That $U\times 1\ra G'\times M$ is the kernel is clear, as kernels are stable
    under taking products. For $G'\times M\ra
    G'/U\times M$ being its cokernel, note that in the diagram \[\begin{tikzcd}
        U\times 1\arrow{r} & G'\times M\arrow{rd}\arrow{r} & G'/U\times
        M\arrow[dashed]{d}\\
& & D
      \end{tikzcd}\]
    the object $G'/U\times M$ is discrete as $\cat{F}$ commutes with finite limits, so
    the corresponding proposition in the category of (abstract) monoids yields
    the proposition by \cref{prop:discrete-quotient-is-set-quotient}.
  \end{proof}
\end{proposition}

\begin{remark}\label{rmk:image-coimage-are-horrible-defs}
  A morphism is called \emph{strict}, if its image and coimage
  coincide.
  
  Consider the following notion, which we will call the \emph{classical image},
  which is often simply called the image of a morphism, cf.~\cite[I.10]{MR0202787}:
  The classical image of a morphism $f\colon X\ra Y$ is a monomorphism $CI\inj Y$
  and a morphism $X\ra CI$ such that $f = X\ra CI\inj Y$ and for
  every other factorisation $f = X\ra D\inj Y$ there is a unique
  morphism $CI\ra D$ such that the obvious diagrams commute. We can analogously define the
  \emph{classical coimage} of a morphism.

  Note that it is easy to see that in the category of topological spaces, the
  classical image is the set theoretic image with the quotient topology (i.\,e.,
  $V\subseteq f(X)$ is open if and only if $f^{-1}(V)$ is), while the classical
  coimage is the set theoretic image with the subspace topology of the codomain.

  These notions have to be strictly differentiated from the notions of
  \emph{regular images} and \emph{regular coimages}, which are often simply
  called the image and coimage of a morphism, cf.~\cite[definition
  5.1.1]{MR2182076}: The regular image of a morphism $f\colon X\ra Y$ is defined
  as the equaliser $\lim Y\rightrightarrows Y\sqcup_X Y$ and its regular coimage
  as the coequaliser $\colim X \times_Y X \rightrightarrows X$.

  It is again easy to see that in the category of topological spaces, the
  regular image of a morphism is the set theoretic image with the subspace
  topology, and that the regular coimage is given by the set theoretic image
  with the quotient topology, i.\,e., in the category of topological spaces the
  classical image is the regular coimage and the classical coimage is the
  regular image!

  Indeed, a number of sources simply call regular coimages \emph{images} to make
  the confusion complete. For this reason we decided to avoid the notion
  altogether.
\end{remark}
\section{Rigidified $G$-Modules}
\label{sec:rigidified-g-modules}
Let $\cat{C}$ be a topological category and $G$ a topologised monoid in
$\cat{C}$. Then we can define a $G$-module as an abelian group object $A$ in
$\cat{C}$ together with a morphism $G\times A\ra A$ subject to the usual
diagrams. Regrettably this definition is too restrictive for our applications.

\begin{definition}
  Let $\cat{C}$ be a topological category and $\cat{D}$ a concrete category. A
  rigidification from $\cat{C}$ to $\cat{D}$ is a bifunctor \[h\colon
    \cat{C}^\circ\times\cat{D}\ra\cat{Set}\] such that functorially in $X$ and
  $Y$, \[ h(X,Y)\subseteq \Hom_{\cat{Set}}(\forget X,\forget Y)\]
  and \[h(\cat{F}\bullet,-)\cong\forget.\]
\end{definition}
\begin{example}\label{ex-tbl:rigidifications}
  Even though we haven't yet defined the notion of $h$-pliant objects, we want
  to give an overview of the most important examples of rigidifications.

\begin{tabular}{p{0.24\textwidth}lp{0.24\textwidth}p{0.24\textwidth}}
    \toprule
    $\cat{C}$ & $\cat{D}$ & $h(X,Y)$ & $h$-pliant objects \\
    \midrule
    any topological category & $\cat{C}$ & $\Hom_{\cat{C}}(X,Y)$ & all discrete objects\\
    analytic manifolds & LF-spaces & locally analytic maps\par in the sense of \cite[section 5]{MR3522263} & all discrete spaces (considered as zero-dimensional manifolds) \\
    topological spaces & metric spaces & bounded continuous maps & finite discrete spaces\\
    \bottomrule
  \end{tabular}
\end{example}
But even this notion of rigidifications is in some cases to restrictive.

\begin{definition}\label{def:set-w-c-rigidification}
  Let $\cat{C}$ be a topological category. A \emph{set with
    $\cat{C}$-rigidification} is a set $Y$ and a contravariant functor
  $h_Y\colon \cat{C}^\circ\ra\cat{Set}$ such that functorially in $X\in\cat{C}$, \[
    h_Y(X)\subseteq \Hom_{\cat{Set}}(\forget X,Y).\]
  We furthermore require that $h_Y(\cat{F}\bullet)=Y.$
  
  For $f\in h_Y(X)$ we will also write $f\colon X\sra Y$. If for
  discrete $D$ and all $X$ we have an equality
  $h_Y(D\times X)=\Hom_{\cat{Set}}(\forget D,h_Y(X)),$ we say that $D$ is
  $Y$-pliant. It follows that if $D$ is $Y$-pliant, then
  $h_Y(D)=\Hom_{\cat{Set}}(\forget D, Y)$.
\end{definition}
\begin{remark}\label{rmk:fck-lf}
  LF-spaces and induced modules are the main reason we have to consider sets
  with rigidifications and not just rigidifications: Assume a group $G$ with
  normal subgroup $N$ was to act on an LF-space $A$ in a suitable sense. Then
  $A^N,$ being a kernel, need not be an LF-space itself, cf.~\cite{MR0075542}.
  But we still have an object with $\cat{C}$-rigidification in the sense of
  \ref{def:set-w-c-rigidification}.
\end{remark}

\begin{remark}
  Let $h$ be a rigidification from $\cat{C}$ to $\cat{D}$.
   Any object in $\cat{D}$ then gives rise to an object with
   $\cat{C}$-rigidification via $Y\mapsto(\forget Y,h(-,Y))$.
\end{remark}
\begin{definition}
  Let $h$ be a rigidification from $\cat{C}$ to $\cat{D}.$ A discrete object $D$
  in $\cat{C}$ is called $h$-pliant, if for all $Y\in\cat{D}$, $D$ is
  $(Y,h(-,Y))$-pliant.
\end{definition}
\begin{definition}\label{def:module-w-rigidification}
  Let $\cat{C}$ be a topological category and $G$ a topologised monoid in
  $\cat{C}$. A $G$-module with $\cat{C}$-rigidification is a set with
  $\cat{C}$-rigidification  $(A,h_A)$ together with a $\forget G$-module
  structure on $A$ such that functorially in $X\in\cat{C}$
  \begin{itemize}
  \item $h_A(X)$ is a subgroup of $\Hom_{\cat{Set}}(\forget X,A)$
  \item for $f\in h_A(X)$ the induced map \[\begin{tikzcd}
        \forget G\times \forget X\arrow{r}{(\id, f)} & \forget G\times
        A\arrow{r}{\mu} & A\end{tikzcd}\] 
    lies in $h_A(G\times X).$
  \end{itemize}

  A morphism of $G$-modules with $\cat{C}$-rigidification $(A,h_A)\ra (B,h_B)$
  is a morphism of functors $h_A\ra h_B$ such that the induced map
  $A = h_A(\cat{F}\bullet) \ra h_B(\cat{F}\bullet) = B$ is a morphism of
  $\forget G$-modules. A sequence \[(A,h_A) \ra (B,h_B)\ra (C,h_C) \] is called
  a short exact sequence of $G$-modules with $\cat{C}$-rigidification, if for
  all $X\in\cat{C}$ the sequence of abelian groups \[ 0 \ra h_A(X) \ra h_B(X)
    \ra h_C(X) \ra 0\] is exact.
\end{definition}

\begin{remark}
  Let $G$ be a topologised group in a topological category $\cat{C}$ and $A$ a
  $G$-module, i.\,e., an abelian group object in $\cat{C}$ with a morphism
  $G\times A\ra A$ subject to the usual diagrams. Then $(\forget A,
  \Hom_{\cat{C}}(-,A))$ is a $G$-module with $\cat{C}$-rigidification.

  If conversely $A$ is an object in $\cat{C}$ and $(\forget A, h_A)$ a
  $G$-module with $\cat{C}$-rigidification, then we can in general only recover
  a morphism $G\times A\ra A$ in $\cat{C}$ if $h_A =
  \forget(\Hom_{\cat{C}}(-,A))$: In this case, the map \[
    \begin{tikzcd}
      \forget G \times \forget A \arrow{r}{(\id,\forget\id)} & \forget
      G\times\forget A\arrow{r}{\mu}& \forget A
    \end{tikzcd}
\]
lies in $\forget(\Hom_{\cat{C}}(G\times A,A)).$
\end{remark}
\begin{remark}
  \label{rmk:strictness-exactness-rig-g-modules}
 The definition of an exact sequence of rigidified $G$-modules has concrete
interpretations in practice, as the following proposition shows. Indeed they
boil down to the usual requirements as for example in \cite[(2.7.2)]{MR2392026}.

The same arguments also show that in a topological category, a sequence of
$G$-modules is exact if it is strict
(cf.~\ref{rmk:image-coimage-are-horrible-defs}) and the last morphism admits a
section in $\cat{C}$.
\end{remark}
\begin{proposition}
  Let $\cat{C}$ be the category of compactly generated weakly Hausdorff spaces
  and $G$ a group object in $\cat{C}$. We fix $G$-modules $A,B,C$ with
  corresponding rigidifications $h_A,h_B,$ and $h_C$. Then the short exact
  sequences \[(A,h_A) \ra (B,h_B)\ra (C,h_C) \] are in one-to-one correspondence
  with exact sequences \[ 0 \ra A \ra B \ra C \ra 0,\] where all maps are
  continuous, $A$ carries the subspace topology of $B$ and there is a continuous
  section $C\ra B.$
  \begin{proof}
    Let us start with a short exact sequences \[(A,h_A) \ra (B,h_B)\ra
      (C,h_C). \] By Yoneda, morphisms $h_A\ra h_B$ are given by morphisms $A\ra
    B$ etc. Evaluating at $\cat{F}\bullet$ hence gives a short exact sequence\[
      0 \ra A \ra B \ra C \ra 0\] with continuous maps. As by assumption \[
      \Hom_{\cat{C}}(C,B) = h_B(C)\ra h_C(C) =\Hom_{\cat{C}}(C,C)\] is
    surjective and the latter includes the identity, this yields a section.

    We also clearly have a continuous bijective map $\iota\colon A\ra \iota(A)$,
    where the latter carries the subspace topology. The inclusion
    $\iota(A)\subseteq A$ clearly get mapped to zero in $h_C(\iota(A))$, so has
    to come from an element in $h_A(\iota(A)),$ which is the (continuous)
    inverse to $\iota.$

    Conversely, if we start with a short exact sequence \[ 0 \ra A \ra B \ra C
      \ra 0\] with all maps continuous, $B\ra C$ admitting a section and $A$
    carrying the subspace topology, it is easy to see that indeed all
    sequences \[ 0 \ra h_A(X) \ra h_B(X) \ra h_C(X) \ra 0\] are exact.
  \end{proof}
\end{proposition}
\begin{definition}\label{def:invariants-of-rigidified-module}
  Let $\cat{C}$ be a topological category, $G$ a topologised monoid in $\cat{C}$
  and $(A,h_A)$ a $G$-module with $\cat{C}$-rigidification. For a normal
  subgroup $N\leq G$ we define $A^N=(A^{\forget N}, h_{A^N})$ by \[ h_{A^N}(X) =
    \left\{ f\in h_A(X)\mid f(\forget X)\subseteq A^{\forget N} \right\}.\] We
  immediately see that this is again a $G$-module with $\cat{C}$-rigidification.
\end{definition}
\begin{remark}
  Let $\cat{C}$ be a topological category and $G$ a topologised group with
  normal subgroup $N$. Let $A$
  be a $G$-module in the sense that $A$ is an abelian group object in $\cat{C}$
  together with a morphism $G\times A\ra A$ subject to the usual diagrams.

  Then there is a slightly more natural notion of the invariants $A^N$: Let
  $g\in\forget G$. Then \ref{prop:top-cat-has-const-morph} yields a
  morphism \[m_g\colon A \ra \cat{F}\bullet\times A\ra G\times A\ra A\] that we
  call multiplication by $g$. We will also denote
  $m_g-\id_A\in\Hom_{\cat{C}}(A,A)$ by $g-1.$

  For a finite set $R\subseteq \forget G,$ we denote by \[A^R =
    \ker\begin{tikzcd}A\arrow{rr}{(g-1)_{g\in R}} && \prod_{g\in R}
      A.\end{tikzcd} \] and clearly \[\forget (A^R)=\forget(A)^R.\]

  If $\cat{C}$ admits arbitrary limits, we can analogously define $A^G.$ If
  there is a finite set $R\subseteq \forget G$ such that $\forget(A)^{\forget
    G}=\forget(A)^R,$ then we will also call $A^G=A^R$ and it is an easy
  exercise that both definitions of $A^G$ (when applicable) coincide. In this
  case, the universal property of the kernel yields an action of $G$ on $A^N.$
  In the presence of a section $G/N\ra G$ in $\cat{C}$, we also get a morphism
  $G/N\times A^N\ra A^N$ and we can check on the level of sets that this gives
  $A^N$ the structure of a $G/N$-module.

  It is easy to check that both definitions of invariants coincide, i.\,e., \[
    (\forget (A^N), \Hom_{\cat{C}}(-,A^N)) = ( (\forget A)^{\forget N}, h_{A^N}),\]
  where $h_{A^N}$ is defined as in \cref{def:invariants-of-rigidified-module}.
\end{remark}

\section{The Induced Module}
\label{sec:induced-module}
Let $\cat{C}$ be a topologised category, $G$ a topologised monoid in $\cat{C}$
and $H$ a submonoid of $G$. Let $(A,h_A)$ be an $H$-module with
$\cat{C}$-rigidification.

\begin{definition}
  \begin{gather*}\Ind_G^H(A)\colon\cat{C}^\circ\ra\cat{Set}\\ X\mapsto
    \left\{f\in h_A(X\times G)\mid f(x,hg)=h.f(x,g) \text{ for all } x\in\forget
    X, h\in\forget H,g\in\forget G\right\}\end{gather*}
  is called the induced module of $A$ from $H$ to $G$.
\end{definition}

\begin{proposition}
  Set $I=\Ind_G^H(A)$, then
  $I$ is a $G$-module with $\cat{C}$-rigidification, if we give the set
  $I(\cat{F}\bullet)\subseteq h_A(G)$ the $\forget G$-module action of right
  translation: \[ (gf)(\sigma) = f(\sigma g) \text{ for } f\in
    I(\cat{F}\bullet), g,\sigma\in\forget G.\]
  \begin{proof}
    The only difficulty lies in the formalism.

    Note first that for $g\in\forget G$ and $f\in I(\cat{F}\bullet),$ $gf$
    indeed lies in $h_A(G),$ as right-multiplication by $g$ is a morphism in
    $\cat{C}.$ It then follows immediately that $gf\in I(\cat{F}\bullet).$

    We have to show that for $f\in I(X)\subseteq \Hom_{\cat{Set}}(\forget
    X, I(\cat{F}\bullet))\subseteq \Hom_{\cat{Set}}(\forget X, h_A(G))$ the induced
    map \[\begin{tikzcd}
        \forget G\times \forget X\arrow{r}{(\id, f)} & \forget G\times
        I(\cat{F}\bullet)\arrow{r}{\mu} &
        I(\cat{F}\bullet)\end{tikzcd}\]
    lies in $I(G\times X)\subseteq h_A(G\times X \times G)$.
    
    For this note that there is a morphism
    $G\times X\times G\ra X\times G$ in $\cat{C}$, which on the level of
    elements is given by $(g,x,g')\mapsto (x,g'g)$.
    Precomposing with this morphism yields a map \[ h_A(X\times G) \ra h_A(G
      \times X \times G)\] and it is evident that under this map, the subset
    $I(X)$ gets sent into $I(G\times X)$, which is precisely the map we need.
  \end{proof}
\end{proposition}
\begin{remark}
  If $\cat{C}$ is the category of analytic manifolds over a non-archimedean
  field, there is also a less sheafy view on the subject of induced modules: Let
  $G$ be a group object in $\cat{C}$, $H\leq G$ a closed subgroup, and $A$ an
  analytic representation of $H$. Then there is a
  natural topology on the induced module $\Ind_G^H(A)$, such that the action of
  $G$ on $\Ind_G^H(A)$ is itself analytic, cf.~\cite[Kapitel 4]{MR1691735}.
\end{remark}

We have now set the stage to define the cohomology of topologised monoids with
coefficients in a rigidified module. Our aim for the remainder of this article
is to prove some standard results of group cohomology in this setting. Namely,
we compare cohomology of topologised monoids with their discrete counterparts in
\ref{prop:c-coh-vs-abstract-coh}, show the existence of a long exact sequence in
\ref{prop:l-e-s}, prove two versions of Shapiro's lemma in
\ref{prop:shapiro-top-grps,prop:shapiro-top-mon}, and show variants of the
classical Hochschild-Serre spectral sequence in
\ref{prop:topological-hochschild-serre,prop:quasi-iso-in-direct-product-case}.

\section{Abstract Monoid Cohomology}
\label{sec:abstract-monoid-cohomology}

Note first that the cohomology of monoids is trickier than one might expect.

\begin{definition}
  \label{def:abstract-inhomogeneous-coboundary-operator}
  Let $M$ be an abstract monoid. Then we define the standard resolution as
  follows: Denote by $F_n$ the free $\Z[M]$-module with basis $M^n$ and define
  the coboundary operator via \begin{gather*}\partial\colon F_{n+1}\ra F_{n},\\
      \begin{split}(x_1,\dots,x_{n+1})\mapsto &
      x_1\cdot (x_2,\dots, x_{n+1}) +  (-1)^{n+1}(x_1,\dots,x_n) \\ &+ \sum_{i=1}^n
      (-1)^{i}(x_1,\dots,x_{i-1},x_ix_{i+1},x_{i+1},\dots, x_{n+1}).
    \end{split}\end{gather*}
\end{definition}

\begin{proposition}\label{prop:free-resolution-of-z-is-resolution}
  \[ \dots \ra F_2 \ra F_1 \ra F_0 \ra \Z\ra 0\] is a free resolution of the integers.

  \begin{proof}
    This works exactly the same way it does for groups.
  \end{proof}
\end{proposition}

\begin{proposition}
  Let $M$ be an abstract monoid and $A$ an $M$-module. Then the inhomogeneous
  cochain complex computes the cohomology of $A\mapsto A^M$.
  \begin{proof}
   Immediate from \cref{prop:free-resolution-of-z-is-resolution}.
  \end{proof}
\end{proposition}

\begin{remark}
  \label{rmk:homogeneous-cochains-not-free}
  The homogeneous cochains do not necessarily form a free resolution. Indeed,
  set $M=(\Z/2,\cdot)$. Then the homogeneous complex is given by
  $F'_n=\Z[M^{n+1}]$ with diagonal action and the usual differential. However,
  it is \emph{not} a free resolution of the integers: It is evident that $F'_1$
  is not cyclic. But every two elements $e_1,e_2\in F'_1$ admit a non-trivial
  combination of zero: Multiplied by the monoid element $(0)$, both are
  contained in $(0,0)\Z\subseteq \Z[M^2]=F'_1$, say, $(0)\cdot
  e_1=\alpha\cdot(0,0)$ and $(0)\cdot e_2=\beta\cdot(0,0)$. If $\alpha$ or
  $\beta$ is zero, this is a non-trivial combination of zero, otherwise
  $\beta\cdot(0)\cdot e_1-\alpha\cdot(0)\cdot e_2$ will do.

  Nonetheless, $F'_1$ is still a \emph{projective} $\Z[M]$-module. Consider the
  $\Z$-linear homomorphisms $F'_1\ra \Z[M]$ \begin{gather*}
    A_1\colon (1,1)\mapsto (1); (0,1), (1,0), (0,0) \mapsto (0)\\
    A_2\colon (1,0)\mapsto (1)-(0); (1,1), (0,1), (0,0)\mapsto 0\\
    A_3\colon (0,1)\mapsto (1)-(0); (1,1), (1,0), (0,0)\mapsto 0.
  \end{gather*}
  It is easy to verify that these maps are actually $\Z[M]$-linear and that \[ x
    \mapsto (1,1) A_1(x) + (1,0) A_2(x) + (0,1) A_3(x)\] is the identity on
  $F'_1,$ which by the dual basis theorem is therefore projective. One can
  analogously show that $F'_\bullet$ is still a projective resolution of $\Z$ in
  this case.
\end{remark}

\section{Setup}
\label{sec:setup-topological-cochains}
For the remainder of this article, we fix a topological category
$\cat{C}=(\cat{C},\forget,\cat{F})$ in which everything takes place, a
topologised group $G'$ with an open normal subgroup $U'$ and abelian topologised
monoids $M_1,\dots, M_r$. Set \[ U = U' \times M_1^{e_1}\times\dots\times
  M_r^{e_r}\] with all $e_i\in\{0,1\}$ and $e_i=1$ if $M_i$ is not discrete. We
also set $G=G'\times \prod_i M_i$ and see that $U\ra G$ has a cokernel $G/U
\cong G'\times\prod_i M_i^{1-e_1},$ whose kernel is $U$ and which is discrete,
cf.~\cref{prop:quotient-well-behaved-after-product-with-monoid}. We will
furthermore use the shorthand $M=\prod_i M_i.$

We let further $N'$ be a normal subgroup of $G'$ and $N=N'\times\prod_i M_i^{e_i'}$
with $e_i'\in\{0,1\}$. It is again evident that $N\ra G$ has a cokernel (which we
denote by $G/N$) and that the kernel of this cokernel is precisely $N$. We
furthermore require the existence of a section $s\colon G/N\ra G$ in $\cat{C}$
whose image on the level of sets contains the neutral element. If $N=U,$ this
exists automatically by \cref{prop:discrete-quotient-has-section} and everything
that we prove for the $N$ will also automatically be true for $U$, but the
converse does not hold. The section is of vital importance; without it,
statements such as $\forget(G/N)\cong\forget G/\forget N$ need not be true,
cf.~\ref{prop:section-implies-quotient-is-set-quotient}.

The projections $G\ra G/N$ and $G\ra G/U$ will both be denoted by $\pi$. It will
always be clear from the context which map is meant.

The section gives rise to two important morphisms: On the one hand, the
\emph{choice of a representative} morphism
$(-)^*\colon G\ra G$, which is the composition of the projection onto $G/N$
followed by the section $s\colon G/N\ra G$, and on the other hand the morphism
$(-)_N\colon G\ra G$ which on $\forget G$ is given by $x\mapsto (x^*)^{-1}x$
with the obvious interpretation of this on the monoid parts (either the identity or
constant $1$). It is clear that $(-)_N$ factors through $N\ra G,$ and that the
composition \[ \begin{tikzcd}N\arrow{r}& G \arrow{r}{(-)_N} & N\end{tikzcd}\] is
the identity on $N$. Evidently we can factor the identity on $G$
as \[\begin{tikzcd} G \arrow{r}{\Delta} & G\times G\arrow{rr}{((-)^*, (-)_N)}& &
    G\times G\arrow{r}{\mu} & G.\end{tikzcd}\] It is important that these maps
exist in $\cat{C},$ which is why we spell out these details.

We also fix a $G$-module with $\cat{C}$-rigidification $A=(A,h_A)$ and assume
that $G/U$ is $A$-pliant, cf.\,\ref{ex-tbl:rigidifications} for examples of what
this means in practice.

For $n\geq 0$ set
  \begin{itemize}
  \item $X^n=X^n(G,A)=h_A(G^n) \subseteq
    \Hom_{\cat{Set}}(\forget (G)^n, A)$, the inhomogeneous cochains, with the
    convention that $G^0\cong\cat{F}\bullet$ and hence $X^0=A,$
  \item $C^n=C^n(G,A)$ those maps $f\in X^n(G,A)$ such that $f(x_1,\dots,x_n)=0$
    if at least one of the $x_i$ is $1$, the normalised cochains,
  \item $I^jC^n$ those maps in $C^n$ that come from morphisms in $h_A(G^{n-j}
    \times (G/N)^j)$ (i.\,e., $I^jC^n= C^n \cap h_A(G^{n-j}
    \times (G/N)^j)$, the intersection taking place in $X^n$). We set $I^jC^n=0$
    for $j>n.$
  \end{itemize}

  We can characterise the filtration as follows:
  \begin{lemma}\label{prop:n-invariance-enough-for-i-filt}
    Let $f\in C^n.$ Then $f\in I^jC^n$ if and only if the last $j$ arguments are
    $\forget N$-invariant, i.\,e., \[ f(x_1,\dots,x_{n-j}, x_{n-j+1} \sigma_1,\dots,
      x_n \sigma_j) = f(x_1,\dots, x_n) \text{ for all } x_i\in \forget G,
      \sigma_i\in \forget N.\]
    \begin{proof}
      The ``only if'' part of the proposition is clear.

      For the ``if'' part, we only show the case of $j=n=1,$ as the other cases
      follow completely analogously.

      As we have a section $s\colon G/N\ra G$,
      \ref{prop:section-implies-quotient-is-set-quotient} shows that 
      $\forget G/\forget N\cong\forget (G/N)$. Consider the following
      commutative diagram:
      \[
        \begin{tikzcd}
          h_A(G/N) \arrow{d}{h_A(\pi)}\arrow[hook]{r}\arrow[bend right=80,swap]{dd}{\id} &
          \Hom_{\cat{Set}}(\forget G/\forget N, A)
          \arrow{d}{\Hom_{\cat{Set}}(\forget \pi, A)} \arrow[bend left=80]{dd}{\id}\\
          h_A(G)\arrow[hook]{r}\arrow{d}{h_A(s)} & \Hom_{\cat{Set}}(\forget
          G,A)\arrow{d}{\Hom_{\cat{Set}}(\forget s, A)}\\
          h_A(G/N) \arrow[hook]{r}& \Hom_{\cat{Set}}(\forget G/\forget N, A)
        \end{tikzcd}
      \]
      Starting with an $\forget N$-invariant $f\in h_A(G)$, we know from the
      assumption that it comes from an element in $\Hom_{\cat{Set}}(\forget
      G/\forget N, A)$. But the diagram implies that this element is necessarily
      identical to $h_A(s)(f)$.
    \end{proof}
  \end{lemma}

  Note that for $f\in X^n,$ the induced face maps \[s_k(f)\colon(x_i)\mapsto
  f(x_1,\dots,x_k,1,x_{k+1},\dots,x_{n-1})\] lie in $X^{n-1}$.

  We will often omit the forgetful functor $\forget $, e.\,g.,
  instead of $x\in \forget N$ we will simply write $x\in N.$
  
  \begin{proposition}
    \label{prop:coboundary-operator-well-defined}
    The assignment
    \begin{align*}
      \partial f(x_1,\dots,x_{n+1}) &= x_1.f(x_2,\dots,x_{n+1}) + (-1)^{n+1}f(x_1,\dots,x_n)\\
      & \quad\quad + \sum_{i=1}^n (-1)^i f(x_1,\dots,x_{i-1},x_ix_{i+1},x_{i+2},\dots, x_{n+1}),
    \end{align*}
    induced by \ref{def:abstract-inhomogeneous-coboundary-operator}, gives rise to
    well-defines maps
    \begin{gather*}
      \partial\colon X^n\ra X^{n+1},\\
      \partial\colon C^n\ra C^{n+1},\\
      \partial\colon I^jC^j\ra I^jC^{n+1}.
    \end{gather*}
    \begin{proof}
      To see that $\partial f\in X^{n+1}$ for $f\in X^n,$ it suffices to check
      this for each of the summands, as $X^{n+1}$ is an abelian group per
      definition. All but the first summand stem from a composition \[
        \begin{tikzcd}
         G^{n+1} \arrow{r} & G^n\arrow[sra]{r}{f} & A 
        \end{tikzcd}
        \] and hence lie in $X^{n+1}$. That the
    first summand lies in $X^{n+1}$ is precisely the second requirement in
    \cref{def:module-w-rigidification}.

    If $f$ is normalised, then in the expansion of $\partial
    f(x_1,\dots,x_i,1,x_{i+1},\dots,x_n)$ there are exactly two terms that are
    not trivially zero. But these terms are identical except for an opposing
    sign and hence cancel.
    
    If $f\in I^jC^n,$ we want to show that the last $j$ arguments of $\partial
    f$ are $\forget N$-invariant. But the last $j$ arguments of $\partial f$
    only contribute to the last $j$ arguments of every individual summand in the
    coboundary expansion of $\partial f$, which are $\forget N$-invariant.
    \end{proof}
  \end{proposition}

  \begin{remark}
    It might seem artificial to consider cochains whose \emph{last} $j$
    arguments are $\forget N$-invariant instead of the \emph{first} $j$, which
    will also lead to somewhat counter-intuitive definitions later on. But the
    equation \[\partial f(x,y) = x.f(y) - f(xy) + f(x)\] shows that if $f$ is
    $\forget N$-invariant, only the second argument of $\partial f$ is $\forget
    N$-invariant and not necessarily the first.
  \end{remark}

  \begin{remark}
    \label{rmk:homogeneous-cochains-for-groups}
    For $G'$ we can also form the complex $ \widetilde{X}^\bullet(G',A)$ given
    by \[ \widetilde{X}^n(G',A) = h_A((G')^{n+1})\] with differential \[
      \widetilde{\partial}(f)(x_0,\dots,x_{n+1}) = \sum_{i=0}^{n+1}(-1)^i
      f(x_0,\dots,x_{i-1},x_{i+1},\dots,x_{n+1}).\] By the usual arguments
    (cf.~e.\,g.~\cite[14]{MR2392026}), we get an isomorphism of complexes \[
      X^\bullet(G',A)\cong\widetilde{X}^\bullet(G',A).\] However, the usual
    morphism \[\phi\colon X^n(G',A)\ra \widetilde{X}^n(G',A),\] which is given by
    \[ \phi(f)(x_0,\dots,x_n) = x_0
      f(x_0^{-1}x_1,x_1^{-1}x_2,\dots,x_{n-1}^{-1}x_n)\] can only be defined for
    the group object $G'$ and not for the monoid object $G.$ The example in
    \ref{rmk:homogeneous-cochains-not-free} shows that both complexes cannot be
    isomorphic in general. They might still be quasi-isomorphic, however, we
    were unable to show this.
  \end{remark}
  
\section{Cohomology of Topologised Monoids}
\label{sec:cohomology-of-topologised-monoids} 
We will call $ H^\bullet(G,A) = H^\bullet(X^\bullet,\partial)$ the
($\cat{C}$-)cohomology of $G$ with coefficients in $A$. Note that if $G$ is
$A$-pliant, the $\cat{C}$-cohomology of $G$ with coefficients in $A$ is just the
(abstract) monoid cohomology of $\forget G$ with coefficients in $A.$ Generally,
comparing topological cohomology with abstract cohomology only works well in low
degrees.
\begin{proposition}
  \label{prop:c-coh-vs-abstract-coh}
  For every $n$ there is a natural morphism \[ H^n(G,A)\ra H^n(\forget G,A),\]
  which is an isomorphism for $n=0$ and injective for $n=1$.
  \begin{proof}
    Clearly the following diagram commutes
    \[\begin{tikzcd}
        X^{n-1}(G,A) \arrow{r}{\partial}\arrow[hook]{d}& X^n(G,A) \arrow{r}{\partial}\arrow[hook]{d} & X^{n+1}(G,A)\arrow[hook]{d}\\
        X^{n-1}(\forget G,A) \arrow{r}{\partial}& X^n(\forget G,A)
        \arrow{r}{\partial} & X^{n+1}(\forget G,A),
      \end{tikzcd}
      \] which yields the required comparison morphisms. By definition, \[
        X^0(G,A)=X^0(\forget G,A),\] so the morphism is indeed an isomorphism
      for $n=0$ and injective for $n=1.$
  \end{proof}
\end{proposition}

Our definition of an exact sequence of rigidified $G$-modules is custom tailored
to admit a long exact sequence of cohomology groups.

\begin{theorem}
  \label{prop:l-e-s}
  Let \[ 0 \ra A \ra B \ra C \ra 0\] be an exact sequence of rigidified
  $G$-modules.
  Then there is a long exact sequence of abelian groups \[\begin{tikzcd}
      0 \arrow{r} & A^{\forget G} \arrow{r} & B^{\forget G} \arrow{r} &
      C^{\forget G} \arrow{r} & \dots \arrow[out=355, in=175, looseness=0.9, overlay]{dlll} \\
      & H^n(G,A) \arrow{r} & H^n(G,B) \arrow{r} & H^n(G,C) \arrow[out=355, in=175, looseness=0.9, overlay]{dll}\\
      & H^{n+1}(G,A)\arrow{r} & \dots
    \end{tikzcd}\]
  \begin{proof}
    By definition of exactness of a sequence of rigidified $G$-modules, we have
    the following commutative diagram with exact rows:
    \[
      \begin{tikzcd}
        & \vdots\arrow{d}{\partial} & \vdots \arrow{d}{\partial} & \vdots \arrow{d}{\partial} & \\
        0 \arrow{r} & X^n(G,A)\arrow{r}\arrow{d}{\partial} & X^n(G,B)\arrow{r}\arrow{d}{\partial} & X^n(G,C) \arrow{r}\arrow{d}{\partial}& 0 \\
        0 \arrow{r} & X^{n+1}(G,A) \arrow{r}\arrow{d}{\partial}& X^{n+1}(G,B)\arrow{r}\arrow{d}{\partial} &\arrow{r}\arrow{d}{\partial} X^{n+1}(G,C) \arrow{r}\arrow{d}{\partial}& 0 \\
        & \vdots & \vdots & \vdots &
      \end{tikzcd}
    \]
    As usual, the snake lemma implies the existence of the long exact sequence
    as required.
  \end{proof}
\end{theorem}

As in the classical case, the normalised cochains compute the same cohomology as
all inhomogeneous cochains.

\begin{proposition}
  \label{prop:c-quasiiso-x}
  The inclusion $C^\bullet\ra X^\bullet$ is a quasi-isomorphism.
  \begin{proof}
    The original proof in \cite[§\,6]{MR0019092} works without issues, as for
    every $f\in X^n$ the map $s_kf\colon (x_1,\dots,x_{n-1})\mapsto
    f(x_1,\dots,x_{k-1},1,x_k,\dots,x_{n-1})$ is again in $X^{n-1}$.  \end{proof}
\end{proposition}

\section{Cohomology and Extensions}
\label{sec:cohomology-and-extensions}
For discrete coefficients, the groups $H^i(G,A)$ have concrete interpretations
for $i\leq 2.$ We will give two concrete interpretations for $H^1$ also in the
topological case. For this matter, we fix in this section another topological
category $\cat{D}$ and a rigidification \[h\colon \cat{C}^\circ \times
  \cat{D}\ra\cat{Set}.\]

In this section, we assume that $A$ is actually an abelian group object in
$\cat{D}$ and that $h_A=h(-,A).$

\begin{definition}
  An $A$-torsor is an object $X$ in $\cat{D}$ with a right action from $A$
  (i.\,e., an arrow $\mu\colon X\times A\ra X$ in $\cat{D}$ subject to the usual
  conditions), such that the induced map \[ \begin{tikzcd}
      m\colon X\times A\arrow{r}{(\id,\mu)} & X\times
    X\end{tikzcd}\] is an isomorphism. The composition $\pi_A\circ m^{-1}\colon
X\times X\ra A$ will be denoted by $\backslash.$

  An $A$-torsor with $G$-rigidification is an $A$-torsor $X$, such that
  $(\forget X,h(-,X))$ is a $G$-set with $\cat{C}$-rigidification, and on the
  level of sets for all $g\in\forget G,x\in\forget X,$ and $a\in\forget A$ the
  following holds:
  \[ g\mu(x,a) = \mu(gx,ga).\]
  An isomorphism of $A$-torsors with $G$-rigidification $j\colon X\ra Y$ is an
  isomorphism in $D$ such that on the level of sets, $j$ commutes with both the
  $A$- and $G$-action.
\end{definition}

\begin{theorem}
  \label{prop:h1-torsors}
  $H^1(G,A)$ stands in one-to-one correspondence with isomorphism classes of
  $A$-torsors with $G$-rigidification.
  \begin{proof}
    Given a torsor $X$, we construct an element in $H^1(G,A)$ as follows: First
    choose $x\in\forget X.$ The definition of a $G$-set, together with the
    existence of constant maps in topological categories, imply the existence of
    $\cdot x\in h(G,X),$ which on the level of sets is just given by $g\mapsto
    g\cdot x.$ As $h$ is a bifunctor, we can compose it as follows:
    \[\begin{tikzcd}
        G \arrow[sra]{r}{\cdot x}\arrow[sra]{rdd}[description]{c_X} & X\arrow{d}{(x,\id)}\\
        & X \times X\arrow{d}{\backslash} \\
        & X
        \end{tikzcd}
      \]
      Note that on the level of sets, $c_X(g)$ is the unique element such that
      $g\cdot x= xc_X(g).$ The verification that $c_X$ is a well-defined cocycle
      independent of $x\in \forget X$ is standard.
      
      For the other direction, take a cocycle in $H^1(G,A)$, represented by
      $c\colon G\sra A.$ Define $X=A$ and $\mu\colon X\times A\ra X$ as the
      addition in $A$. We define the $\forget G$-action on $\forget X$
      via \[g.x = c(g) + g\cdot x,\] where $g\cdot x$ is the given action of $G$
      on $A$. It is easy to check that this gives a well defined $G$-module with
      $\cat{C}$-rigidification. The verifications that this construction is (up
      to isomorphism) independent of the class of $c$ and both left and right
      inverse to the previous construction is again standard.
  \end{proof}
\end{theorem}

Let $R$ be a ring object in $\cat{D},$ which is furthermore commutative and
unitary. For the remainder of this section we assume that $A$ is an $R$-module,
i.\,e., we additionally require the existence of a morphism $R\times A\ra A$
subject to the usual conditions. We also assume that $G$ operates on $R$ via
ring endomorphisms in $\cat{D}$ and that on the level of sets, the action on $A$
is $R$-semi-linear, i.\,e., \[ g\cdot(r\cdot a) = (g\cdot r)\cdot (g\cdot a)
  \text { for all } g\in\forget G,r\in\forget R, a\in\forget A.\] 
We will call such an object a semi-linear $G$-module over $R$.

An exact sequence \[ 0 \ra M'\ra M\ra M''\ra 0\] of semi-linear $G$-modules is
an exact sequence of $G$-modules with $\cat{C}$-rigidification in the sense of
\ref{def:module-w-rigidification}, where we additionally require the
morphisms \[ 0 \ra h_{M'}(X) \ra h_M(X) \ra h_{M''}(X)\ra 0\] to be
$R$-linear.
In the above exact sequence, $M$ will be called an extension of $M''$ by $M'$.

An equivalence of extensions of $M''$ by $M'$, which will be denoted by
$M\approx\widetilde{M}$, is an isomorphism of functors
$h_M\oset{\cong}{\ra}h_{\widetilde{M}}$ such that for all $X$ the following
diagram commutes:
\[
  \begin{tikzcd}
    0 \arrow[r] & h_{M'}(X) \arrow[r] \arrow[equal]{d} & h_M(X) \arrow[r] \arrow[d, "\cong"] & h_{M''}(X) \arrow[r] \arrow[equal]{d} & 0 \\
0 \arrow[r] & h_{M'}(X) \arrow[r]           & h_{\widetilde{M}}(X) \arrow[r]      & h_{M''}(X) \arrow[r]           & 0.
  \end{tikzcd}
\]
The equivalence classes of extensions of $M''$ by $M'$ will be denoted by $\Ext(M'',M').$

\begin{theorem}
  \label{prop:h1-extensions}
  $H^1(G,A)\cong\Ext(R,A).$
  \begin{proof}
    Because of the similarity to \ref{prop:h1-torsors}, we only sketch the
    construction.
    
    Let \[ 0 \ra A \ra E \oset{p}{\ra} R \ra 0\] be an extension and denote by
    $1\in\forget R$ the unit in $R$. We can construct a cochain via \[ g \mapsto
      g\cdot e-e,\] where $e\in \forget E$ is any preimage of $1$.

    On the other hand, for a cochain $c\colon G\sra A$ define a $G$-action on
    $A\times R$ via \[ g\cdot(a,r) = ((g\cdot r)\cdot c(g) + g\cdot a,g\cdot
      r),\] which is a well-defined semi-linear $G$-module over $R$. The
    universal property of the product yields the exactness.
 \end{proof}
 \end{theorem}
 \begin{remark}
   An alternative to the construction of \ref{prop:h1-extensions} goes as
   follows: As $\cat{D}$ is a topological category and all limits exist, we get
   an object and a morphism $X=p^{-1}(1)\ra E.$ It follows that we have an
   action $X\times A\ra X$ given by addition in $E$ and that the composition
    \[
      \begin{tikzcd}
        X \times A \arrow{r}{(\id,+)} & X \times X \arrow{rr}{(\id,x_2-x_1)} &&
        X\times E
      \end{tikzcd}
    \] factors through a morphism $X\times A\ra X\times A$ and induces the identity,
    i.\,e., $X$ is a $A$-torsor. It is also evident that $X$ inherits a
    $G$-rigidification from $E$.

    Directly
    constructing an extension from a torsor $X$ is regrettably not
    straight-forward, as it is very cumbersome to define the correct $R$-module
    structure on $X\times R$.
 
 \end{remark}
\section{Shapiro's Lemma for Topologised Groups}
\label{sec:shapiro-top-grps}
\begingroup
\allowdisplaybreaks

We will now prove Shapiro's lemma for the induction of subgroups, i.\,e., we
will assume in this section that $M=1$. Later on we will also prove
Shapiro's lemma for actual monoids,
cf.~\ref{sec:shapiro-for-monoids}.

Let $H\leq G$ be a subgroup of $G$ and assume the existence of a map \[
  {}_H(-)\colon G \ra H\] in $\cat{C}$ with the following properties:

\begin{itemize}
\item ${}_H(1)=1$,
\item ${}_H(hg)=h\cdot{}_H(g)$ for all $h\in H,g\in G$
\end{itemize}

\begin{remark}
  Instead of requiring the existence of such a morphism in $\cat{C},$ one could
  also construct it as follows: Assume that the push-out
  $H\backslash G$ of the following diagram exists:
  \[
\begin{tikzcd}
H \times G \arrow[r, "\pi_2"] \arrow[d, "\mu"] & G \arrow[d, "p"] \\
G \arrow[r]                                    & H\backslash G   
\end{tikzcd}
  \]
  $H\backslash G$ is then the space of right cosets of $H$ in $G$. Assume the
  existence of a section $s\colon H\backslash G\ra G$ in $\cat{C}$ with $p\circ
  s=\id_{H\backslash G}$ and with the neutral element contained in the image of
  $s$. We would then define \[ {}_H(g) = g s(p(g))^{-1}.\] Note the similarity
  with $(-)_N,$ which was defined as $ (g)_N = s(\pi(g))^{-1}g.$ As we wanted to
  use the same convention for the action of $G$ on $\Ind^H_G(A)$ as in the
  literature, we have to use a different convention here. However, we
  would then need to check many basic properties of this construction, which
  wouldn't shed any additional light on what actually happens.
\end{remark}

\begin{example}
  These requirements are always satisfied if $G$ is an analytic group and $H$ a
  closed subgroup, cf.~\cite[section III.1.6]{MR979493}.
\end{example}

\begin{definition}
  Consider the following maps:
  \begin{enumerate}
    \item $\alpha_n\colon C^n(G,\Ind^H_G(A)) \ra C^n(H,A)$ given by
      \[\alpha_n(f)(h_1,\dots, h_n) = f(h_1,\dots,h_n,1).\] Note that our
      formalism ensures that this is a well-defined map.
    \item $\beta_n\colon C^n(H,A)\ra C^n(G,\Ind^H_G(A))$  by
      \begin{align*}\beta_n(f)(g_1,\dots,g_n,x)=&{}_H(x) f({}_H(x)^{-1}{}_H(xg_1),
                                                  {}_H(xg_1)^{-1}{}_H(xg_1g_2), \dots,\\
        &\quad\quad\quad\quad{}_H(xg_1\dots
        g_{n-1})^{-1}{}_H(xg_1\dots g_n)).\end{align*} As we can express
      $\beta_n(f)$ in a (very large) diagram in $\cat{C},$ it again gives a
      well-defined element in $h_A(G^n\times G)$, and we immediately verify that
      it indeed lies in $C^n(G,\Ind^H_G(A)).$
    \item $\kappa_{n+1}\colon C^{n+1}(G,\Ind^H_G(A)) \ra C^n(G,\Ind^H_G(A))$ given
      by \begin{align*}
           \MoveEqLeft[3]
           \kappa_{n+1}f(g_1,\dots,g_n,x)\\
           \begin{split}
             ={}& f(x^{-1}{}_H(x), {}_H(x)^{-1}{}_H(xg_1),\\
             &\quad\quad {}_H(xg_1)^{-1}{}_H(xg_1g_2),\dots, {}_H(xg_1\dots g_{n-1})^{-1}{}_H(xg_1\dots g_n),x)\\
           & + \sum_{i=1}^n (-1)^i f(g_1,\dots,g_{i},
           (xg_1\dots g_i)^{-1}{}_H(xg_1\dots g_{i}),\\
           &\quad\quad\quad\quad\quad\quad {}_H(xg_1\dots
           g_i)^{-1}{}_H(xg_1\dots g_{i+1}), \dots,\\
            &\quad\quad\quad\quad\quad\quad {}_H(xg_1\dots g_{n-1})^{-1}
        {}_H(xg_1\dots g_{n}),x).
      \end{split}
         \end{align*}
         The summands of $\kappa_{n+1}$ are given by morphisms in $\cat{C}$
         followed by $f,$ so indeed $\kappa_{n+1}(f)\in X^n(G,\Ind^H_G(A))$ and
         also in $C^n(G,\Ind^H_G(A)).$
  \end{enumerate}

\end{definition}

\begin{remark}
  While we can adapt the definition of $\beta$ to also work in a monoid setting
  by explaining what the map is supposed to do on the monoid part, this is no
  longer true for $\kappa.$ We will prove later in \ref{sec:shapiro-for-monoids}
  that Shapiro's lemma still holds in the monoid setting.
\end{remark}

The proof of Shapiro's lemma now consists of the
following few lemmata which show that $\alpha_\bullet$ and $\beta_\bullet$ are
quasi-isomorphisms. Their proofs are routine, excruciatingly unenlightening, and
given only for sake of completeness.

\begin{lemma}
  $\alpha_\bullet$ and $\beta_\bullet$ are maps of chain complexes, i.\,e., they
  commute with $\partial$.
  \begin{proof}
    First,
    \begin{align*}
      \MoveEqLeft[6]
      \partial \alpha_n f(h_1,\dots, h_{n+1}) \\
        ={}& h_{1} \alpha_n f (h_2,\dots, h_{n+1}) + (-1)^{n+1} \alpha_n f (h_1,\dots, h_n) \\
           & + \sum_{i=1}^n (-1)^i \alpha_n f(h_1,\dots,h_{i-1}, h_ih_{i+1},h_{i+2},\dots,h_{n+1}) \\
        ={}& h_{1} f(h_2,\dots, h_{n+1}, 1) + (-1)^{n+1} f (h_1,\dots, h_n, 1) \\
           & + \sum_{i=1}^n (-1)^i f(h_1,\dots,h_{i-1}, h_ih_{i+1},h_{i+2},\dots,h_{n+1},1) \\
        ={}& f(h_2,\dots, h_{n+1}, 1\cdot h_{1}) + (-1)^{n+1} f (h_1,\dots, h_n, 1) \\
        & + \sum_{i=1}^n (-1)^i f(h_1,\dots,h_{i-1}, h_ih_{i+1},h_{i+2},\dots,h_{n+1},1) \\
        ={}& \partial f (h_1,\dots,h_{n+1}, 1)\\
        ={}& \alpha_{n+1}\partial f (h_1,\dots,h_{n+1}).
    \end{align*}

    Secondly,
    \begin{align*}
      \MoveEqLeft[6]
      \partial \beta_n f (g_1,\dots, g_{n+1}, x) \\
        ={}& \beta_n f(g_2,\dots,g_{n+1}, xg_1) + (-1)^{n+1} \beta_n f(g_1,\dots,g_n,x)\\
           & + \sum_{i=1}^n (-1)^i \beta_n f(g_1,\dots,g_{i-1},g_ig_{i+1},g_{i+2},\dots, g_{n+1},x) \\
           ={}& {}_H(xg_1)f\bigl({}_H(xg_1)^{-1}{}_H(xg_1g_2),{}_H(xg_1g_2)^{-1}{}_H(xg_1g_2g_3),\dots,\\
           & \quad\quad\quad\quad\quad {}_H(xg_1\dots g_n)^{-1}{}_H(xg_1\dots g_{n+1})\bigr) \\
           & + \sum_{i=1}^n (-1)^i {}_H(x) f\bigl( {}_H(x)^{-1}{}_H(xg_1), {}_H(xg_1)^{-1}{}_H(xg_1g_2),\dots,\\
           & \quad\quad\quad\quad\quad\quad\quad\quad\quad {}_H(xg_1\dots g_{i-1})^{-1}{}_H(xg_1\dots g_{i+1}),\\
           & \quad\quad\quad\quad\quad\quad\quad\quad\quad {}_H(xg_1\dots g_{i+1})^{-1}{}_H(xg_1\dots g_{i+2}),\dots, \\
           & \quad\quad\quad\quad\quad\quad\quad\quad\quad {}_H(xg_1\dots g_n)^{-1}{}_H(xg_1\dots g_{n+1}))\bigr) \\
           & + (-1)^{n+1} {}_H(x) f({}_H(x)^{-1}{}_H(xg_1),\dots,{}_H(xg_1\dots g_{n-1})^{-1}{}_H(xg_1\dots g_n)) \\
           ={}& {}_H(x) \biggl( {}_H(x)^{-1}{}_H(xg_1)f(\dots)+\sum_{i=1}^n (-1)^if(\dots) + (-1)^{n+1}f(\dots) \biggr) \\
           ={}& {}_H(x) \partial f \bigl({}_H(x)^{-1}{}_H(xg_1), {}_H(xg_1)^{-1}{}_H(xg_1g_2),\dots,\\
           &\quad\quad\quad\quad\quad{}_H(xg_1\dots g_n)^{-1}{}_H(xg_1\dots g_{n+1})\bigr) \\
           ={}& \beta_{n+1}\partial f (g_1,\dots,g_{n+1},x).
    \end{align*}
  \end{proof}
\end{lemma}

\begin{lemma}
  $\alpha_\bullet\circ\beta_\bullet = \id_{C^\bullet(H,A)}$.
  \begin{proof}
    Note that ${}_H(-)$ restricted to $H$ is the identity. We hence have
    \begin{align*}
      \alpha_n\beta_n f (h_1,\dots,h_n) ={}& \beta_nf(h_1,\dots, h_n, 1) \\
      \begin{split}
        ={}& {}_H(1) f\bigl({}_H(1)^{-1}{}_H(1\cdot h_1), {}_H(1\cdot h_1)^{-1}{}_H(1\cdot h_1h_2), \dots,\\
        &\quad\quad\quad\quad\quad{}_H(1\cdot h_1\dots h_{n-1})^{-1} {}_H(1\cdot h_1\dots h_n)\bigr) \\
      ={}& f(h_1,\dots, h_n).
      \end{split}
    \end{align*}
  \end{proof}
\end{lemma}

\begin{lemma}
  $\partial\circ\kappa_n + \kappa_{n+1}\circ\partial = \beta_n\circ\alpha_n -
  \id_{C^n(G,\Ind^H_G(A))},$ i.\,e., $\kappa_\bullet$ is a chain homotopy from
  $\beta_\bullet\circ\alpha_\bullet$ to the identity.
  \begin{proof}
    This is going to be as bad as it looks. Let us first compute
    $\partial\circ\kappa_n$ and $\kappa_{n+1}\circ\partial$, subtract
    $(\beta_n\circ\alpha_n - \id)$ from this and show that the sum of the
    remaining terms is zero. We first compute $\partial\circ\kappa_n$:
    \begin{align*}
      \MoveEqLeft[6]
      (\partial \circ \kappa_n) f (g_1,\dots, g_n, x) \\
      \begin{split}
        ={}& \kappa_n(f)(g_2,\dots,g_n,xg_1) + (-1)^n\kappa_n(f)(g_1,\dots,g_{n-1},x) \\
         & + \sum_{i=1}^{n-1} (-1)^i \kappa_n(f)(g_1,\dots,g_{i-1},g_ig_{i+1}, g_{i+2},\dots,g_n,x). \\
      \end{split}
    \end{align*}
    We can also expand \begin{align*} \MoveEqLeft[3]
                         \tag{$\cancer$}\kappa_{n}f(g_2,\dots,g_n,xg_1)\\
                         ={}  &f((xg_1)^{-1}{}_H(xg_1), {}_H(xg_1)^{-1}{}_H(xg_1g_2),\\
                         &\quad\quad{}_H(xg_1)^{-1}{}_H(xg_1g_2),\dots, {}_H(xg_1\dots g_{n-1})^{-1}{}_H(xg_1\dots g_n),xg_1)\\
                              &+ \sum_{j=1}^n (-1)^j f(g_2,\dots,g_{j+1},
                                (xg_1\dots g_{j+1})^{-1}{}_H(xg_1\dots g_{j+1}),\\
                              &\quad\quad\quad\quad\quad\quad\quad\quad
                                {}_H(xg_1\dots
                                g_{j+1})^{-1}{}_H(xg_1\dots g_{j+2}), \dots,\\
                              &\quad\quad\quad\quad\quad\quad\quad\quad
                                {}_H(xg_1\dots g_{n-1})^{-1} {}_H(xg_1\dots
                                g_{n}),xg_1),
                         \\
                         \MoveEqLeft[6]
                         (-1)^n\kappa_{n}f(g_1,\dots,g_{n-1},x)\\
                         ={}\tag{$\Diamond$}& (-1)^n f(x^{-1}{}_H(x), {}_H(x)^{-1}{}_H(xg_1),\\
                         &\quad\quad\quad\quad{}_H(xg_1)^{-1}{}_H(xg_1g_2),\dots, {}_H(xg_1\dots g_{n-2})^{-1}{}_H(xg_1\dots g_{n-1}),x)\\
                              &\tag{$\star$} + \sum_{j=1}^{n-1} (-1)^{j+n}
                                f(g_1,\dots,g_{j},
                                (xg_1\dots g_{j})^{-1}{}_H(xg_1\dots g_{j}),\\
                              & \quad\quad\quad\quad\quad\quad\quad\quad
                                {}_H(xg_1\dots
                                g_{j})^{-1}{}_H(xg_1\dots g_{j+1}) \dots,\\
                              &\quad\quad\quad\quad\quad\quad\quad\quad
                                {}_H(xg_1\dots g_{n-2})^{-1} {}_H(xg_1\dots
                                g_{n-1}),x)
  \end{align*}
  and
  \begin{align*} \MoveEqLeft[4]
             (-1)^i \kappa_n(f)(g_1,\dots,g_{i-1},g_ig_{i+1}, g_{i+2},\dots,g_n,x)\\
    ={} &\tag{$\maltese.i$} (-1)^{i} f(x^{-1}{}_H(x), {}_H(x)^{-1}{}_H(xg_1),{}_H(xg_1)^{-1}{}_H(xg_1g_2),\dots,\\
   &\quad\quad \quad\quad{}_H(xg_1\dots g_{i-2})^{-1}{}_H(xg_1\dots g_{i-1}),\\
   & \quad\quad\quad\quad{}_H(xg_1\dots g_{i-1})^{-1}{}_H(xg_1\dots g_{i+1}),\\
   & \quad\quad\quad\quad{}_H(xg_1\dots g_{i+1})^{-1}{}_H(xg_1\dots g_{i+2}),\dots,\\
   & \quad\quad\quad\quad{}_H(xg_1\dots g_{n-1})^{-1}{}_H(xg_1\dots g_n),x)\\
    & \tag{$\LEFTCIRCLE.i$} + \sum_{j=1}^{i-1} (-1)^{i+j} f(g_1,\dots,g_{j},
           (xg_1\dots g_{j})^{-1}{}_H(xg_1\dots g_{j+1}),\\
           &\quad\quad\quad\quad\quad\quad\quad {}_H(xg_1\dots
             g_{j+1})^{-1}{}_H(xg_1\dots g_{j+2}), \dots,\\
           &\quad\quad\quad\quad\quad\quad\quad  {}_H(xg_1\dots g_{i-2})^{-1}
        {}_H(xg_1\dots g_{i-1}),\\
  &\quad\quad\quad\quad\quad\quad\quad {}_H(x g_1\dots g_{i-1})^{-1}{}_H(xg_1\dots g_{i+1}),\\
           &\quad\quad\quad\quad\quad\quad\quad {}_H(xg_1\dots g_{i+1})^{-1}{}_H(xg_1\dots g_{i+2})^{-1}, \dots,\\
           &\quad\quad\quad\quad\quad\quad\quad  {}_H(xg_1\dots g_{n-1})^{-1}
        {}_H(xg_1\dots g_{n}),x)\\
        %&\tag{$\blacktriangle.i$} + (-1)^{i+i}  f(g_1,\dots, g_{i-1}, g_ig_{i+1}, (x g_1\dots g_{i+1})^{-1}{}_H(xg_1\dots g_{i+1}),\\
        %   &\quad\quad\quad\quad\quad\quad {}_H(xg_1\dots g_{i+1})^{-1}{}_H(xg_1\dots g_{i+2}),\dots,\\
%&\quad\quad\quad\quad\quad\quad {}_H(xg_1\dots g_{n-1})^{-1}{}_H(xg_1\dots g_{n}),x)\\
           &\tag{$\blacktriangleleft.i$} + \sum_{j=i}^{n-1} (-1)^{i+j} f(g_1,\dots, g_{i-1}, g_ig_{i+1},g_{i+2},\dots, g_{j+1},\\
 &\quad\quad\quad\quad\quad\quad\quad\quad   (xg_1\dots g_{j+1})^{-1}{}_H(xg_1\dots g_{j+1}),\\
&\quad\quad\quad\quad\quad\quad\quad\quad{}_H(xg_1\dots g_{j+1})^{-1}{}_H(xg_1\dots g_{j+2}),\dots,\\
&\quad\quad\quad\quad\quad\quad\quad\quad{}_H(xg_1\dots g_{n-1})^{-1}{}_H(xg_1\dots g_{n}),x ).
  \end{align*}
  On the other hand,\begin{align*}
      \MoveEqLeft[6]
      (\kappa_{n+1}\circ\partial) f (g_1,\dots, g_n, x) \\
      \begin{split}
        ={}& \partial(f)(x^{-1}{}_H(x), {}_H(x)^{-1}{}_H(xg_1),\\
        & \quad\quad\quad {}_H(xg_1)^{-1}{}_H(xg_1g_2),\dots, {}_H(xg_1\dots g_{n-1})^{-1}{}_H(xg_1\dots g_n),x)\\
        &+  \sum_{j=1}^n (-1)^j \partial(f)(g_1,\dots,g_j,(xg_1\dots g_j)^{-1}{}_H(xg_1\dots g_j),\\
        &\quad\quad\quad\quad\quad\quad\quad        {}_H(xg_1\dots g_j)^{-1}{}_H(xg_1\dots g_{j+1}), \dots,\\
        &\quad\quad\quad\quad\quad\quad\quad {}_H(xg_1\dots g_{n-1})^{-1}{}_H(xg_1\dots g_n),x)
      \end{split}
    \end{align*}
    and
     \begin{align*}
      \MoveEqLeft[10]
      (-1)^j \partial(f)(g_1,\dots,g_j,(xg_1\dots g_j)^{-1}{}_H(xg_1\dots g_j),\\
      \MoveEqLeft[10]
      \quad\quad
      \quad\quad
      \quad\quad
      {}_H(xg_1\dots g_j)^{-1}{}_H(xg_1\dots g_{j+1}), \dots,\\
      \MoveEqLeft[10]
      \quad\quad
      \quad\quad
      \quad\quad
      {}_H(xg_1\dots g_{n-1})^{-1}{}_H(xg_1\dots g_n),x)\\
       ={}&\tag{$\cancer.j$} (-1)^j f(g_2,\dots,g_j,(xg_1\dots g_j)^{-1}{}_H(xg_1\dots g_j),\\
       &\quad\quad\quad\quad{}_H(xg_1\dots g_j)^{-1}{}_H(xg_1\dots g_{j+1}), \dots,\\
       &\quad\quad\quad\quad{}_H(xg_1\dots g_{n-1})^{-1}{}_H(xg_1\dots g_n),xg_1)\\
       & \tag{$\blacktriangleright.j$}+ \sum_{i=1}^{j-1} (-1)^{i+j} f(g_1,\dots,g_{i-1}, g_ig_{i+1}, g_{i+2},\dots, g_j,\\
&\quad\quad\quad\quad\quad\quad\quad\quad(xg_1\dots g_j)^{-1}{}_H(xg_1\dots g_j),\\
&\quad\quad\quad\quad\quad\quad\quad\quad
{}_H(xg_1\dots g_j)^{-1}{}_H(xg_1\dots g_{j+1}),\dots,\\
&\quad\quad\quad\quad\quad\quad\quad\quad
{}_H(xg_1\dots g_{n-1})^{-1}{}_H(xg_1\dots g_n),x)\\
&\tag{\leftmoon$.j$}+ (-1)^{j+j} f(g_1,\dots, g_{j-1}, (xg_1\dots g_{j-1})^{-1}{}_H(xg_1\dots g_j),\\
       & \quad\quad\quad\quad\quad\quad{}_H(xg_1\dots g_j)^{-1}{}_H(xg_1\dots g_{j+1}), \dots,\\
      & \quad\quad\quad\quad\quad\quad{}_H(xg_1\dots g_{n-1})^{-1}{}_H(xg_1\dots g_n),x)\\
&\tag{\leftmoon$'.j$} + (-1)^{j+j+1} f(g_1,\dots, g_{j}, (xg_1\dots g_{j})^{-1}{}_H(xg_1\dots g_{j+1}),\\
& \quad\quad\quad\quad\quad\quad\quad{}_H(xg_1\dots g_{j+1})^{-1}{}_H(xg_1\dots g_{j+2}), \dots,\\
& \quad\quad\quad\quad\quad\quad\quad{}_H(xg_1\dots g_{n-1})^{-1}{}_H(xg_1\dots g_n),x)\\
&\tag{$\RIGHTCIRCLE.j$} + \sum_{i=j+2}^{n}(-1)^{i+j}f(g_1,\dots,g_j,(xg_1\dots g_j)^{-1}{}_H(xg_1\dots g_j),\\
& \quad\quad\quad\quad\quad\quad\quad\quad{}_H(xg_1\dots g_j)^{-1}{}_H(xg_1\dots g_{j+1}), \dots,\\
& \quad\quad\quad\quad\quad\quad\quad\quad{}_H(xg_1\dots g_{i-3})^{-1}{}_H(xg_1\dots g_{i-2}),\\
& \quad\quad\quad\quad\quad\quad\quad\quad{}_H(xg_1\dots g_{i-2})^{-1}{}_H(xg_1\dots g_{i}),\\
& \quad\quad\quad\quad\quad\quad\quad\quad{}_H(xg_1\dots g_{i})^{-1}{}_H(xg_1\dots g_{i+1}), \dots,\\
& \quad\quad\quad\quad\quad\quad\quad\quad {}_H(xg_1\dots g_{n-1})^{-1}{}_H(xg_1\dots g_n),x)\\
&\tag{$\star.j$}+ (-1)^{j+n+1}f(g_1,\dots,g_j,(xg_1\dots g_j)^{-1}{}_H(xg_1\dots g_j),\\
       & \quad\quad\quad\quad\quad\quad{}_H(xg_1\dots g_j)^{-1}{}_H(xg_1\dots g_{j+1}),\dots,\\
      & \quad\quad\quad\quad\quad\quad{}_H(xg_1\dots g_{n-2})^{-1}{}_H(xg_1\dots g_{n-1}),x).
     \end{align*}
     We furthermore expand
     \begin{align*}
       \MoveEqLeft[5]
      \partial(f)(x^{-1}{}_H(x), {}_H(x)^{-1}{}_H(xg_1),{}_H(xg_1)^{-1}{}_H(xg_1g_2),\dots, \\
       \MoveEqLeft[5]
       \quad\quad\quad {}_H(xg_1\dots g_{n-1})^{-1}{}_H(xg_1\dots g_n),x)\\
       ={}&\tag{$\circ$}  f({}_H(x)^{-1}{}_H(xg_1), {}_H(xg_1)^{-1}{}_H(xg_1g_2),\dots,\\
       & \quad\quad\quad {}_H(xg_1\dots g_{n-1})^{-1}{}_H(xg_1\dots g_n),{}_H(x))\\
       & \tag{$\Diamond'$}+ (-1)^{n+1}f(x^{-1}{}_H(x), {}_H(x)^{-1}{}_H(xg_1), {}_H(xg_1)^{-1}{}_H(xg_1g_2),\dots,\\
       & \quad\quad\quad\quad\quad {}_H(xg_1\dots g_{n-2})^{-1}{}_H(xg_1\dots g_{n-1}),x)\\
       & \tag{\textpilcrow}- f(x^{-1}{}_H(xg_1), {}_H(xg_1)^{-1}{}_H(xg_1g_2),\dots,\\
       & \quad\quad\quad {}_H(xg_1\dots g_{n-1})^{-1}{}_H(xg_1\dots g_n),x)\\
       & \tag{$\maltese$}+ \sum_{i=2}^{n} (-1)^i f(x^{-1}{}_H(x), {}_H(x)^{-1}{}_H(xg_1),\dots,\\
       & \quad\quad\quad\quad\quad\quad {}_H(xg_1\dots g_{i-3})^{-1}{}_H(xg_1\dots g_{i-2}),\\
       & \quad\quad\quad\quad\quad\quad {}_H(xg_1\dots g_{i-2})^{-1} {}_H(xg_1\dots g_i), \\
       & \quad\quad\quad\quad\quad\quad {}_H(xg_1\dots g_i)^{-1}{}_H(xg_1\dots g_{i+1}),\dots,\\
       & \quad\quad\quad\quad\quad\quad {}_H(xg_1\dots g_{n-1})^{-1}{}_H(xg_1\dots g_n), x).
     \end{align*}
     Clearly $(\circ)=\beta_n\circ\alpha_n(f)(g_1,\dots,g_n,x)$ and
     $(\star.n)=-f(g_1,\dots,g_n,x),$ so it remains to show that the other
     summands amount to zero.

     Note first that 
     \begin{align*}
       (\leftmoon.1) &= -(\text{\textpilcrow}),\\
       (\leftmoon'.j) &= -(\leftmoon.j+1)\quad\quad\quad\quad \text{ for } j=1,\dots,n-1,\\
       (\leftmoon'.n) &= 0,
     \end{align*}
     so in $\kappa_{n+1}\circ\partial$ all $(\leftmoon), (\leftmoon'),$ and
     $(\text{\textpilcrow})$-terms cancel. Furthermore it is immediately evident that
     \begin{gather*}
       (\Diamond) = -(\Diamond'),\\
       \sum_{j=1}^n (\cancer.j) =
       - (\cancer),\\
       \sum_{i=1}^{n-1} (\maltese'.i) = -(\maltese), \text{ and }\\
       \sum_{i=1}^{n-1} (\star.j) = - (\star).
     \end{gather*}
     It remains to show that
     \[
       \sum_{i=1}^{n} (\LEFTCIRCLE.i) = - \sum_{j=1}^n (\RIGHTCIRCLE.j)\]
     and \[
       \sum_{i=1}^{n} (\blacktriangleleft.i) = - \sum_{j=1}^n (\blacktriangleright.j).
     \]
     Write \begin{align*}
             F(i,j) ={}& f(g_1,\dots,g_{j},
           (xg_1\dots g_{j})^{-1}{}_H(xg_1\dots g_{j+1}),\\
           &\quad\quad {}_H(xg_1\dots
             g_{j+1})^{-1}{}_H(xg_1\dots g_{j+2}), \dots,\\
           &\quad\quad  {}_H(xg_1\dots g_{i-2})^{-1}
        {}_H(xg_1\dots g_{i-1}),\\
  &\quad\quad {}_H(x g_1\dots g_{i-1})^{-1}{}_H(xg_1\dots g_{i+1}),\\
           &\quad\quad {}_H(xg_1\dots g_{i+1})^{-1}{}_H(xg_1\dots g_{i+2})^{-1}, \dots,\\
           &\quad\quad  {}_H(xg_1\dots g_{n-1})^{-1}
             {}_H(xg_1\dots g_{n}),x),\\
             G(i,j) ={} &f(g_1,\dots, g_{i-1}, g_ig_{i+1},g_{i+2},\dots, g_{j+1},\\
 &\quad\quad   (xg_1\dots g_{j+1})^{-1}{}_H(xg_1\dots g_{j+1}),\\
&\quad\quad{}_H(xg_1\dots g_{j+1})^{-1}{}_H(xg_1\dots g_{j+2}),\dots,\\
&\quad\quad{}_H(xg_1\dots g_{n-1})^{-1}{}_H(xg_1\dots g_{n}),x ),
           \end{align*}
           then
           \begin{gather*}
             \sum_{i=1}^n (\LEFTCIRCLE.i) = \sum_{i=1}^n \sum_{j=1}^{i-1}
             (-1)^{i+j} F(i,j) = \sum_{1\leq j < i \leq n} (-1)^{i+j} F(i,j),\\
             \sum_{j=1}^n (\RIGHTCIRCLE.j) = \sum_{j=1}^n \sum_{i=j+2}^{n}
             (-1)^{i+j} F(i-1,j) = \sum_{1\leq j < i \leq n} (-1)^{i+j+1}
             F(i,j),\\
             \sum_{i=1}^n(\blacktriangleleft.i) = \sum_{i=1}^n \sum_{j=i}^{n-1}
             (-1)^{i+j} G(i,j) = \sum_{1\leq i \leq j \leq n-1} (-1)^{i+j} G(i,j), \text{ and}\\
             \sum_{j=1}^n(\blacktriangleright.j)=\sum_{j=1}^n \sum_{i=1}^{j-1}
             (-1)^{i+j} G(i,j-1) = \sum_{1\leq i \leq j \leq n-1} (-1)^{i+j-1} G(i,j),
           \end{gather*}
           so indeed their sum amounts to zero.
 \end{proof}
\end{lemma}

We immediately deduce Shapiro's lemma.

\begin{theorem}[Shapiro's lemma]
  \label{prop:shapiro-top-grps}
  In the derived category of abelian groups, \[ C^\bullet(G,\Ind^H_G(A))\cong
    C^\bullet(H,A).\] Especially \[ H^n(G,\Ind^H_G(A))\cong H^n(H,A) \text{ for
      all } n.\]
\end{theorem}

\endgroup

\section{A Hochschild-Serre Spectral Sequence}
\label{sec:hochschild-serre}
We devote this section to proving a Hochschild-Serre spectral sequence in a
rather general fashion. Constructing cochains of course happens in
$\Hom_{\cat{Set}}(\forget G^n,A)$ and is
done by the usual tedious calculations. As $G$ is only a monoid, extra care is
required.
Showing that these cochains stem from
(then necessarily unique) elements in $X^\bullet$ poses an additional difficulty.

The spectral sequence will indeed follow from the following spectral sequence
attached to the filtered complex $I^\bullet C^\bullet.$

\begin{definition}
  As for all $n$ the filtration $I^\bullet C^n$ is a finite filtration and all
  $I^j C^\bullet$ form a subcomplex, there is a $E_1$-spectral sequence \[
    ss(I^\bullet C^\bullet)^{p,q}_1\Longrightarrow H^{p+q}(G,A), \]
  where the $E_1$-terms are defined as \[  ss(I^\bullet C^\bullet)_1^{p,q}=\frac{\ker\left(  
                                        I^pC^{p+q}/I^{p+1}C^{p+q}\oset{\partial}{\ra} I^pC^{p+q+1}/I^{p+1}C^{p+q+1}\right)}{\image\left(
                                        I^pC^{p+q-1}/I^{p+1}C^{p+q-1} \oset{\partial}{\ra} I^pC^{p+q}/I^{p+1}C^{p+q}\right)},\]
  cf.~e.\,g.~\cite[(2.2.1)]{MR2392026}.
\end{definition}

\begin{definition}
  To simplify reading, we will use the following notational convention: The first
  time a variable is used, a superscript will denote the set it belongs to. For
  example, instead of ``Let $\underline{x}\in G^p.$ Then define
  $f(\underline{x})=\ldots$'' we will simply write ``Define
  $f(\mbof{\underline{x}}{G^p})=\ldots$''
\end{definition}

\subsection{The Meaning of $C^i(G/N,C^j(N,A))$}
\label{sec:meaning-of-ci-g-mod-n-cj-n-a}
In the previous section, we fixed a concrete category $\cat{C}$ to encapsulate
our topological data. Our setup allowed us to give meaning to $C^j(N,A)$ for
$G$-modules $A$. It is
however very unclear how we can give the abstract module $C^j(N,A)$ again the
structure of an object in $\cat{C}.$

If $\cat{C}$ is the category of Hausdorff topological spaces, one can topologise
$C^j(N,A)$ with the compact-open topology, which if we further restrict to
compactly generated spaces, has somewhat nice properties and can be called
canonical. But computing cohomology, we are presented with the issue that images
of differentials need not be closed and one subsequently loses the Hausdorff
property, cf.~\cite{MR0352333} for a thorough discussion of these issues.

If $\cat{C}$ is a bit more exotic, e.\,g., analytic $\Q_p$-manifolds, then there
is no obvious way to give $C^j(N,A)$ the structure of an analytic
$\Q_p$-manifold compatible with the additional structure — and especially none
that also correctly topologises the cohomology groups.

If the quotient is discrete, we can identify $C^q(U,A)$ with $\cat{F}C^q(U,A)$
and define \[ C^p(G/U, C^q(U,A)) = C^p(G/U,\cat{F} C^q(U,A)),\] but note that
while $C^q(U,A)$ carries the structure of a $G$-module, in general it does
\emph{not} carry the structure of a $G/U$-module (the action is only $\forget
U$-invariant after passing to cohomology).

\begin{definition}
  For $N=U,$ let $f\in I^p C^{p+q}(G,A)$ and define its $p$-restriction
  $r_p(f)\in C^p(G/U, C^q(U,A))(=C^p(G/U,\cat{F} C^q(U,A)))$ via \[
    r_p(f)(\mbof{x_1}{G/U},\dots,\mbof{x_p}{G/U})(\mbof{y_1}{U},\dots,
    \mbof{y_q}{U}) = f(y_1,\dots,y_q, s(x_1),\dots, s(x_p)).\]
  This is well-defined: The
  induced map \[r_p(f)(\mbof{\underline{x}}{(G/U)^p})\colon U^q\sra A\] stems from the
  composition \[\begin{tikzcd} U^q \arrow{rrr}{(\id, s(x_1),\dots, s(x_p))} & &
      & U^q\times G^p \arrow{r} & G^q\times G^p \arrow[sra]{r}{f} &
      A,\end{tikzcd}\] where we use the existence of constant maps
  (cf.~\cref{prop:top-cat-has-const-morph}). As $G/U$ is discrete,
  $\underline{x}\mapsto r_p(f)(\underline{x})$ is indeed in $C^p(G/U,
  C^q(U,A)).$
\end{definition}

\begin{lemma}
  \label{prop:rp-trivial-if-too-deep}
  If $p'<p$ and $f\in I^p C^{p+q}$ then $r_{p'}(f)=0.$
  \begin{proof}
   \begin{align*} 
      r_{p'}(f)(\mbof{\underline{x}}{(G/U)^{p'}})(\mbof{y_1}{U},\dots,\mbof{y_{p+q-p'}}{U})
     & = f(y_1,\dots, y_q, y_{q+1},\dots, y_{p+q-p'},s(\underline{x}))\\
     &=
       f(y_1,\dots, y_q, 1,\dots, 1,s(\underline{x})) \\
     &= 0,
     \end{align*} as $f$ is normalised by assumption.
  \end{proof}
\end{lemma}

\subsection{Extensions of cochains}\label{sec:extension-of-cochains}
Comparing $C^{p+q}(G,A)$ with $C^p(G/U, C^q(U,A))$, we will have to extend maps
$U^q \times (G/U)^p \ra A$ to maps $G^q\times (G/U)^p\ra A$. This extension
process also works if we work with $N$ instead of $U$.

\begin{definition}\label{def:cochain-ext}
  Let $g\colon G^{k-1} \times N^{q-k} \times G^p \sra A$ be a normalised map
  (meaning its value being zero if one of the arguments is $1$) with $k\geq 2$.
  Then for normalised $f\colon G^k \times N^{q-k}\times G^p\sra A$ we define
  \begin{gather*}
    \ext_f(g) \colon G^{k-2}\times G \times G \times
    N^{q-k-1}\times
    G^p\sra A\\
    (\mbof{\underline{y}}{G^{k-2}}, \mbof{w}{G}, \mbof{x}{G}, \mbof{\underline{\sigma}}{
      N^{q-k-1}}, \mbof{\underline{z}}{G^p})\mapsto g(\underline{y}, wx^*, x_N,
    \underline{\sigma}, \underline{z}) + (-1)^k f(\underline{y},
    w,x^*,x_N,\underline{\sigma},\underline{z})
  \end{gather*}
  It is called the extension of $g$ along $f$ and is again normalised. Note that
  it actually lies in $h_A$, as the modification of the arguments is done via
  morphisms in $\cat{C}$.
\end{definition}

Calling it an extension is due to the following fact which is immediately
verified:

\begin{lemma}
  \label{prop:cochain-ext-is-ext}
  In the setting of \cref{def:cochain-ext} the following
  diagram commutes:
  \[ \begin{tikzcd}
      G^{k-1} \times N \times N^{q-k-1} \times G^p \arrow[hook]{dd}\arrow[sra]{dr}{g} & \\
      & A \\
      G^{k-1} \times G \times N^{q-k-1} \times G^p \arrow[sra]{ru}{\ext_f(g)} & \\
    \end{tikzcd}  \]
  \begin{proof}
    Vectors in the image of the inclusion have $x^*=1$ and $x_N=x.$
  \end{proof}
\end{lemma}

\begin{lemma}
     \label{diag:cochain-extensions-are-extensions-after-partial}
   In the setting of \cref{def:cochain-ext}, the usual coboundary formula gives
   meaning to the function \[
     \partial(\ext_f(g))\colon G^k \times N^{q-k} \times G^p \sra A\] and the
   following diagram commutes:
     \[\begin{tikzcd}
       G^{k-1} \times N \times N^{q-k} \times G^p \arrow[hook]{dd}\arrow[sra]{dr}{\partial(g)} & \\
       & A \\
       G^{k-1} \times G \times N^{q-k} \times G^p \arrow[sra]{ru}{\partial(\ext_f(g))} & \\
     \end{tikzcd}\]
   \begin{proof}
     Immediate from \cref{prop:cochain-ext-is-ext}.
   \end{proof}
\end{lemma}

\begin{remark}
  In the setting of \cref{def:cochain-ext}, $g$ is only defined on $G^{k-1}
  \times N \times N^{q-k-1} \times G^p$. For the coboundary formula to make
  sense, all terms $(x_1,\dots,x_i x_{i+1}, \dots, x_{p+q})$ must lie in $G^{k-1}
  \times N \times N^{q-k-1} \times G^p$. This is the reason $\partial g$ is only
  defined on $G^{k-1} \times N \times N^{q-k} \times G^p$.
\end{remark}

\begin{proposition}\label{prop:cochain-ext-main-work}
  Assume $N=U.$ Let $q\geq 2$ and $f\in I^pC^{p+q}$ with $\partial f\in I^{p+1}C^{p+q+1}$.
  %Assume that $f$ is zero in $C^p(G/U,
  %H^q(U,A)),$ i.\,e., \[r_p(f)(\mbof{\underline{x}}{(G/U)^p})=\partial(u(\underline{x}))\] for
  Take $u\in C^p(G/U, C^{q-1}(U,A))$. Define an element $g=g(u,f)$ as follows:
  \begin{itemize}
  \item $g_0(\mbof{\underline{\sigma}}{U^{q-1}}, \mbof{\underline{y}}{G^p}) =
       u(\underline{y})(\underline{\sigma}),$
     \item $g_1(\mbof{x}{G}, \mbof{\underline{\sigma}}{U^{q-2}}, \mbof{\underline{y}}{G^p}) = x^*.g_0(x_U,
       \underline{\sigma}, \underline{y}) - f(x^*,x_U,\underline{\sigma},
       \underline{y})$,
       \item $g_k = \ext_f(g_{k-1})\in h_A(G^k\times U^{q-1-k}\times G^p)$ for $2\leq k \leq q-1.$
       \item $g=g_{q-1}.$
  \end{itemize}
  Then the following hold:
  \begin{enumerate}
  \item $g\in I^pC^{p+q-1},$
  \item $r_p(g)=u,$
  \item if $r_p(f)(\mbof{\underline{x}}{(G/U)^p})=\partial(u(\underline{x}))$
    for all $\underline{x}\in (G/U)^p,$ then $f-\partial g\in I^{p+1}C^{p+q},$
  \end{enumerate}
  \begin{proof}  
    The first assertion follows immediately from the definitions, as none of the
    manipulations touches the last $p$ arguments.

    The second assertion follows inductively, the details being carried out in
    \cite[proof of theorem 1]{MR0052438} in slightly different phrasing and in
    \cite[proposition 3.6.9]{thomas:on-analytic-and-iwasawa-cohomology} in this
    exact phrasing.
   \end{proof}
\end{proposition}

\begin{remark}
  The cochain $g=g(u,f)$ of \ref{prop:cochain-ext-main-work} has a rather
  unwieldy definition. However, if $f=0$ and $U$ is a direct factor of $G$, then
  $g$ has an explicit description,
  cf.~\ref{prop:explicit-determination-of-ext-by-zero-for-products}.
\end{remark}
\subsection{Comparison of the first page}

\begin{proposition}\label{prop:iso-on-e1-for-closed-subgrp}
  \begin{gather*}
    ss(I^\bullet C^\bullet)^{p,0}_1 \cong C^p(G/N,A^N) 
  \end{gather*}
  for all $p$.
  \begin{proof}
    By definition, $ss(I^\bullet C^\bullet)^{p,0}_1 = \ker I^p C^p\oset{\partial}{\ra} I^p
    C^{p+1}/I^{p+1}C^{p+1}$ and furthermore, $f\in I^pC^p$ comes from a
    (necessarily unique) morphism $(G/N)^p\sra A,$ which yields an element in
    $C^p(G/N, A).$ We first need to show that the image of $f$ is contained in
    $A^N$, i.\,e., \[(\mbof{n}{
        N}-1)f(\mbof{\underline{x}}{G^p})=0.\]
    As $f$ is in the aforementioned kernel, $\partial f(n,\underline{x}) =
    \partial f(1,\underline{x}) = 0$, and
    as $f\in I^pC^p$ by assumption, the difference between the coboundary
    expansions of $\partial f(n,\underline{x})$ and $\partial f(1,\underline{x})$
    is exactly $(n-1)f(\underline{x})$ and hence also zero. Injectivity of the
    map is then clear.

    On the other hand, for $f\in C^p(G/N,A^N)$ coming from
    $\widetilde{f}\in h_{A^N}((G/N)^p),$ consider the induced morphism
    \[
      \begin{tikzcd}
        \widetilde{g}\colon G^p\arrow{r} & (G/N)^p\arrow[sra]{r}{\widetilde{f}} & A^N\arrow{r} &
        A.
      \end{tikzcd}\]
    It is clear that $g\in I^pC^p$ and gets mapped to $f$. As the image of $g$
    lies in $A^N$, $\partial g$ is $\forget N$-invariant for all of its $p+1$
    arguments, so \ref{prop:n-invariance-enough-for-i-filt} implies that
    indeed $g\in ss(I^\bullet C^\bullet)^{p,0}_1$.
  \end{proof}
\end{proposition}

\begin{remark}
  By abuse of notation, even for not necessarily open $N$ we will refer to the
  map $ss(I^\bullet C^\bullet)^{p,0}_1\ra C^p(G/N,A^N)$ of
  \ref{prop:iso-on-e1-for-closed-subgrp} as a map $r_p\colon ss(I^\bullet
  C^\bullet)^{p,0}_1\ra C^p(G/N,H^0(N,A)).$ This is clearly compatible with the
  previous definition of $r_p.$
\end{remark}

\begin{proposition}\label{prop:hochschild-serre-iso-on-e1}
  Suppose that $N=U.$ Then $r_p\colon I^pC^{p+q} \ra C^p(G/U, C^{q}(U,A))$
  induces an isomorphism between the $E_1$-terms: \[ ss(I^\bullet
    C^\bullet)_1^{p,q} \cong C^p(G/U, H^q(U,A)).\]
  \begin{proof}
    Recall that by definition,
    \begin{align*}
      ss(I^\bullet C^\bullet)_1^{p,q}=&\frac{\ker\left(  
                                        I^pC^{p+q}/I^{p+1}C^{p+q}\oset{\partial}{\ra} I^pC^{p+q+1}/I^{p+1}C^{p+q+1}\right)}{\image\left(
                                        I^pC^{p+q-1}/I^{p+1}C^{p+q-1} \oset{\partial}{\ra} I^pC^{p+q}/I^{p+1}C^{p+q}\right)}\\
      =& \frac{\ker \left(I^pC^{p+q} \oset{\partial}{\ra}
         C^{p+q+1}/I^{p+1}C^{p+q+1}\right)}{\partial(I^pC^{p+q-1}) +
         I^{p+1}C^{p+q}}.
    \end{align*}
     We will first prove injectivity. Therefore, take $f\in I^p C^{p+q}$ with
     $\partial f\in I^{p+1}C^{p+q+1}.$ Assume that $f$ is zero in $C^p(G/U,
     H^q(U,A)),$ i.\,e., \[r_p(f)(\mbof{\underline{x}}{(G/U)^p})=\partial(u(\underline{x}))\] for
     some $u\in C^p(G/U, C^{q-1}(U,A))$. We
     want to find an $h\in I^pC^{p+q-1}$ with $f-\partial(h)\in I^{p+1}C^{p+q}.$
   
     The case of $q=0$ was already dealt with in
     \ref{prop:iso-on-e1-for-closed-subgrp}.
     
     If $q=1$, then define $h\in I^pC^{p+1-1}$ as the normalised cocycle
     corresponding to $u\in C^p(G/U, C^0(U,A))=C^p(G/U, A)$. Note that by
     assumption $u$ has the property \[ f(\mbof{x}{G},\mbof{\underline{y}}{G^p})
       = r_p(f)(\underline{y})(x) = \partial(u(\underline{y}))(x)=
       x.u(\underline{y}) - u(\underline{y}).\] We want to show that
     $f-\partial(h)\in I^{p+1}C^{p+1}$. For that matter we need to show that
     \[(f-\partial(h))(\mbof{x}{G}\cdot\mbof{\sigma}{U},\mbof{\underline{y}}{G^p})\]
     is independent of $\sigma.$ As $\partial f\in I^{p+1}C^{p+2}$, we see that
     $\partial(f)(\mbof{x}{G},\mbof{\sigma}{U},\mbof{\underline{y}}{G^p})=0.$
     Expansion of the coboundary operator hence yields \[
       f(x\sigma,\underline{y}) = x.f(\sigma,\underline{y}) +
       f(x,\underline{y}),\] as $f\in I^pC^{p+q}$, and also \[\partial(h)(x\sigma,\underline{y}) =
       x.\sigma.h(\underline{y}) + \text{ terms independent of $\sigma$},\] as
     $h\in I^pC^p.$ We
     hence get
     \begin{align*}
       (f-\partial(h))(x\sigma,\underline{y}) &= f(x\sigma,\underline{y}) - x.\sigma.h(\underline{y}) + \text{ terms independent of $\sigma$}\\
                                              &= x.f(\sigma,\underline{y}) + f(x,\underline{y}) - x.\sigma.h(\underline{y})+ \text{ terms independent of $\sigma$}\\
                                              &= x.(\sigma.u(\underline{y}) - u(\underline{y})) - x.\sigma.u(\underline{y}) + \text{ terms independent of $\sigma$},
     \end{align*}
     which is independent of $\sigma$.
     
     For $q>1$ we are in the situation of \ref{prop:cochain-ext-main-work},
     which deals with exactly this case.

     For surjectivity, take $u\in C^p(G/U,C^q(U,A))$ such that
     $\partial(u(\mbof{\underline{x}}{(G/U)^p}))=0$ for all $\underline{x}$. For
     $q=0$ finding a preimage is trivial, for $q\geq 1$
     \ref{prop:cochain-ext-main-work} yields a preimage $g\in I^pC^{p+q}$ with
     $\partial g\in I^{p+1}C^{p+q+1}$ (take $f=0$ in the proposition).
  \end{proof}
\end{proposition}

\subsection{The shuffling mechanism}

To compare the differential in the spectral sequence attached to $I^\bullet
C^\bullet$, Hochschild and Serre use a process they call \emph{shuffling}.
\begin{lemma}
  \label{prop:existence-of-conjugation}
  Every $x\in\forget G$ induces a conjugation action $G\ra G$ that
  on $M$ is trivial and on $G'$ is the usual conjugation by the $G'$-part of $x$.
  \begin{proof}
    On $G'$ conjugation is defined via the composition \[\begin{tikzcd} G' \cong
        \cat{F}\bullet \times G'\times \cat{F}\bullet \arrow{rr}{(x^{-1},\id,x)}
        && G'\times G' \times G' \arrow{rr}{\textrm{mult}\circ\textrm{mult}} &&
        G',\end{tikzcd}\] where we use the existence of constant maps from
    \ref{prop:top-cat-has-const-morph}.
  \end{proof}
\end{lemma}
\begin{remark}
  We will write formulas like $x^{-1}yx$ even though $x$ need not be invertible
  in $G$. As $M$ is central in $G$, the usual identities such as $y (y^{-1}xy) =
  xy$ still hold.
\end{remark}

\begin{definition}
  \label{def:morphism-ordered-sets}
  We will make use of the ordered sets \[ \ceil{n} = \left\{ 1,\dots, n \right\}\]
  for $n\in\N.$
  For every injective morphism of ordered sets \[\phi\colon
    \ceil{p}\ra\ceil{p+q}\] there exists a unique (injective)
  morphism \[\phi^*\colon\ceil{q}\ra\ceil{p+q}\] such that \[\ceil{p+q} =
    \image\phi \cup\image \phi^*.\]

  We furthermore define \[\sgn{\phi} = (-1)^{\sum_{i=1}^q \phi^*(i)-i}.\]
\end{definition}

\begin{lemma}
  \label{prop:shuffle-sign-identity}
  Let $\phi\colon\ceil{p}\ra\ceil{p+q}$ be an injective morphism of ordered
  sets. Then $\sgn(\phi)\cdot\sgn(\phi^*)=(-1)^{p\cdot q}$.
  \begin{proof}
    \begin{align*}
      \sum_{i=1}^q \phi^*(i)-i + \sum_{i=1}^p \phi(i)-i - p\cdot q & =
                                                                     \frac{(p+q)(p+q+1) - q(q+1) - p(p+1)}{2} - pq\\
                                                                   &= \frac{2pq}{2}-pq=0.
    \end{align*}
  \end{proof}
\end{lemma}

\begin{definition}
  \label{def:shuffle}
  Denote by $F_{p+q}$ the free $\Z[\forget  G]$-module with basis
  $\forget  G^{p+q}.$
  
  For $\phi\colon\ceil{p}\ra\ceil{p+q}$ an injective morphism
  of ordered sets define \[
    (x_1,\dots,x_q,y_1,\dots,y_p)^\phi =
    (\gamma_1,\dots,\gamma_{p+q})\] with \[\gamma_{\phi(i)} = y_i\]
  and \[\gamma_{\phi^*(i)} =
    (y_1\cdots y_{\phi^*(i)-i})^{-1}x_i(y_1\cdots y_{\phi^*(i)-i}).\]
  If we are considering multiple morphisms $\phi$, we will also write
  $\gamma(\phi,k)$ instead of $\gamma_k.$
  
  Define now \[\shuffle_{p}^{p+q}(\mbof{\underline{z}}{\forget G^{p+q}}) =
    \sum_{\phi} \sgn(\phi)\underline{z}^\phi\in F_{p+q}, \] where here and in
  the following an unspecified sum over $\phi$ denotes the sum over all
  injective morphisms of ordered sets with $p$ and $p+q$ clear from the context.

  Every $g\in C^{p+q}$ gives rise to $g^\phi\in C^{p+q}$ via \[
    g^\phi(\underline{z}) = g(\underline{z}^\phi).\] Indeed $g^\phi\in X^{p+q}$ as
  both conjugation and reordering come from morphisms in $\cat{C}$. We can also define
  $\shuffle_p^{p+q}g\in C^{p+q}$ via \[\shuffle_p^{p+q}(g)(\underline{z}) =
    \sum_\phi \sgn(\phi) g(\underline{z}^\phi).\]
  We will use the convention that $\shuffle_0^n=\id$.
\end{definition}
\begin{proposition}\label{prop:shuffling-is-identity-if-deep-enough}
  Let $\phi\colon\ceil{p}\ra\ceil{p+q}$ be the unique injective morphism of
  ordered sets with $\phi(1)=q+1$. Then $\underline{z}=\underline{z}^{\phi}$ for
  all $z\in G^{p+q}$. If $g\in I^pC^{p+q},$ then
  \[\shuffle_p^{p+q}g=g\text{ on } N^q\times G^p.\]
  \begin{proof}
    The first assertion is clear from the definitions, as then $\phi^*(i)=i$ for
    all $1\leq i \leq q$.

    For the second assertion we will show that for all other $\varphi$,
    $g^\varphi=0$ on $N^q\times G^p.$ In this case, there exists $q+1\leq i\leq
    p+q$ with $i=\varphi^*(k)$ for some $k$ and $\gamma_i$ is then equal to a
    conjugate of $\alpha_k,$ which lies by assumption again in $N$. As $g$ was
    supposed to be normalised and $N$-invariant in the last $p$ components, this
    implies that $g^\varphi=0$ on $N^q\times G^p.$
  \end{proof}
\end{proposition}

\begin{definition}\label{def:partial-coboundary-operators}
  For $p,q\geq 1$ we define the following two partial coboundary operators
  $F_{p+q}\ra F_{p+q-1}$: \begin{align*}
    \partial_q(x_1,\dots,x_q,y_1,\dots,y_p) ={} & x_1.(x_2,\dots,x_q,\underline{y})
    + (-1)^q (x_1,\dots,x_{q-1},\underline{y}) \\ &+ \sum_{i=1}^{q-1} (-1)^i
    (x_1,\dots,x_{i-1},x_ix_{i+1},x_{i+2},\dots,x_q,\underline{y})\end{align*}
  and \begin{align*}\delta_p (x_1,\dots,x_q,y_1,\dots,y_p)={}& y_1 .
    (y_1^{-1}\underline{x}y_1, y_2,\dots,y_p) + (-1)^p (\underline{x},
    y_1,\dots, y_{p-1}) \\& + \sum_{i=1}^{p-1} (-1)^k (\underline{x},
    y_1,\dots,y_{i-1}, y_i y_{i+1}, y_{i+2},\dots, y_p),\end{align*} where
  $\underline{x}=(x_1,\dots,x_q), \underline{y} = (y_1,\dots, y_p)$ and
  $y_1^{-1}\underline{x}y_1 = (y_1^{-1}x_1y_1,\dots,y_1^{-1}x_qy_1)$. These
  formulas also give rise to partial coboundary operators
  $\partial_q,\delta_p\colon C^{p+q-1}\ra C^{p+q}$ by the same arguments as in
  \ref{prop:coboundary-operator-well-defined}.
\end{definition}

\begin{proposition}\label{prop:killing-impure-terms}
  Let $\Phi$ denote the set of injective morphisms of ordered sets
  $\ceil{p}\ra\ceil{p+q}.$ For each $1\leq k\leq p+q-1$ there is a
  bijection \[\left\{ \phi\in\Phi \mid k\in\image\phi^*\text{ and }
      k+1\in\image\phi \right\}\lra \left\{ \psi\in \Phi \mid
      k\in\image\psi\text{ and } k+1\in\image\psi^* \right\}\] with the
  following property: If $\phi$ corresponds to $\psi$ then $\gamma(\phi,
  k)\gamma(\phi, k+1) = \gamma(\psi,k)\gamma(\psi,k+1)$ and $\gamma(\phi,
  i)=\gamma(\psi, i)$ for all $i\neq k,k+1.$ Furthermore, $\sgn(\phi) =
  -\sgn(\psi).$
  \begin{proof}
    Construct $\psi$ as follows: Let $k+1=\phi(a)$. Then
    \begin{gather*}
      \psi(1)=\phi(1),\dots,\psi(a-1)=\phi(a-1), \\
      \psi(a)=k,\\
      \psi(a+1)=\phi(a+1),\dots, \psi(p)=\phi(p)
    \end{gather*}
    and conversely for given $\psi$ with $\psi(b)=k$ construct $\phi$ via
    \begin{gather*}
      \phi(1)=\psi(1),\dots,\phi(b-1)=\psi(b-1), \\
      \phi(b)=k+1,\\
      \phi(b+1)=\psi(b+1),\dots, \phi(p)=\psi(p).
    \end{gather*}
    It is clear that both constructions are mutually inverse to one another and
    satisfy above requirements.
\end{proof}
\end{proposition}

\begin{proposition}\label{prop:shuffling-and-partial-coboundary}
  For $p,q\geq 1$ and $\underline{z}\in F_{p+q}$ we have \[
    \partial\shuffle_p^{p+q}(\underline{z}) =
    (\shuffle_p^{p+q-1}\partial_q\underline{z}) + (-1)^q
    (\shuffle_{p-1}^{p+q-1}\delta_p\underline{z}).\] Consequently, for $f\in
  C^{p+q-1}$ the following identity holds: \[ \shuffle_p^{p+q}(\partial f) =
    \partial_q(\shuffle_p^{p+q-1}(f)) + (-1)^q
    \delta_p(\shuffle_{p-1}^{p+q-1}(f)).\]
\end{proposition}

The proof of this is of course a combinatorial nightmare. Details can be found
in \cite[proposition 2]{MR0052438}, and even more details in \cite[proposition
3.6.22]{thomas:on-analytic-and-iwasawa-cohomology}.

\subsection{Comparison of the second page}

So far, we only considered the \emph{groups} $C^p(G/U,H^q(U,A))$ with the
$E_1$-terms corresponding to the spectral sequence attached to $I^\bullet
C^\bullet.$ Now we need to give $C^\bullet(G/U, H^q(U,A))$ the structure of a
\emph{complex}. For this it suffices to give $H^q(U,A)$ the structure of a
$\forget(G/U)$-module, as then $C^\bullet(G/U,
H^q(U,A))=C^\bullet(G/U,\cat{F}H^q(U,A))$ is a complex by
\ref{sec:setup-topological-cochains}. The module-structure also exists for
non-open subgroups $N$.

We also need compatibility between our partial coboundary operators
$\partial_q,\delta_p$ and the operators \[\delta\colon C^p(G/U, \cat{F}C^q(U,A)) \ra
  C^{p+1}(G/U, \cat{F}C^q(U,A))\] and \[ \partial\colon C^p(G/U, \cat{F}C^q(U,A)) \ra
  C^p(G/U, \cat{F}C^{q+1}(U,A)).\] 
\begin{proposition}
  \label{prop:action-on-cochains-and-cohomology}
  $C^q(N,A)$ and $H^q(N,A)$ carry the structure of a $\forget G$-module by the usual
  conjugation action.
  \begin{proof}
    Recall that every element in $y\in\forget G$ induces a conjugation morphism
    on $G$ (\ref{prop:existence-of-conjugation}), which restricts to a morphism
    on $N$. Defining \[ (y.(\mbof{f}{C^q(N,A)}))(\mbof{\underline{x}}{N^q}) =
      y.(f(y^{-1}\underline{x}y))\] yields an element in $C^q(N,A)$ because of
    \ref{def:module-w-rigidification}, so altogether we get a $G$-action on
    $C^q(N,A).$  As in the classical case, this also gives an
    action on the cohomology groups.
  \end{proof}
\end{proposition}
 
\begin{lemma}\label{prop:partial-coboundary-and-rp}
  The diagrams
  \[
    \begin{tikzcd}
      {C^{p+q}(G,A)} \arrow[r, "r_p"] \arrow[d, "\delta_{p+1}"] & {C^p(G/U,C^q(U,A))} \arrow[d, "\delta"] \\
      {C^{p+q+1}(G,A)} \arrow[r, "r_{p+1}"]                   & {C^{p+1}(G/U,C^q(U,A))} 
    \end{tikzcd}
  \] and
  \[
    \begin{tikzcd}
      {C^{p+q}(G,A)} \arrow[r, "r_p"] \arrow[d, "\partial_{q+1}"] & {C^p(G/U,C^q(U,A))} \arrow[d, "\partial"] \\
      {C^{p+q+1}(G,A)} \arrow[r, "r_{p}"]                   & {C^{p}(G/U,C^{q+1}(U,A))} 
    \end{tikzcd}
  \] are commutative. As before, $\delta_{p+1}$ and $\partial_{q+1}$ are the
  partial coboundary operators from \ref{def:partial-coboundary-operators} and
  $\delta$ and $\partial$ are the respective coboundary operators of
  $C^\bullet(G/U,-)$ and $C^\bullet(U,-)$.
  \begin{proof}
    Immediate from the definitions.
  \end{proof}
\end{lemma}
\begin{proposition}
  \label{prop:group-operates-trivially-on-its-cohomology}
  $\forget N$ operates trivially on $H^q(N,A).$
  \begin{proof}
    We use \ref{prop:shuffling-and-partial-coboundary} for the topologised
    monoid $N$: For $p=1$ and $f\in C^q(N,A)\cap\ker\partial$ this reads \[0
      =\shuffle_q^{1+q}(\partial f)= \partial_q(\shuffle_1^q(f)) +
      (-1)^q\delta_1(\shuffle_0^q(f))\] and
    hence \[\delta_1(f)\in\image\partial_q.\] But $\delta_1(f)$ is explicitly
    given by \[ (\delta_1 f)(\mbof{\underline{x}}{N^q},\mbof{y}{N})=
      y.f(y^{-1}\underline{x}y)-f(\underline{x}) =
      (y.f)(\underline{x})-f(\underline{x}),\] so $y.f$ and $f$ are
    cohomologous, as $\partial_q$ is the differential on $C^\bullet(N,A),$
    analogously to \ref{prop:partial-coboundary-and-rp}.
  \end{proof}
\end{proposition}

\begin{theorem}
  \label{prop:topological-hochschild-serre}
  There is a convergent $E_2$-spectral sequence \[ H^p(G/U, H^q(U,A))
    \Longrightarrow H^{p+q}(G,A).\]

  Even if $N$ is not necessarily open, we have the classical five term exact sequence:
  \[ \begin{tikzcd}
      0 \arrow{r} & H^1(G/N,A^N) \arrow{r} & H^1(G,A)\arrow{r} &
      H^1(N,A)^{\forget(G/N)}\arrow[looseness=1.2, overlay, out=355, in=175]{dll} \\
      & H^2(G/N,A^N)\arrow{r} & H^2(G,A).
    \end{tikzcd}\]
  \begin{proof}
    Consider first the case of $N=U.$
    By \cref{prop:hochschild-serre-iso-on-e1} it suffices to show that the
    following diagram commutes:
    \begin{equation}
      \label[diagram]{diag:hs-comparing-e1-differentials}
      \tag{$\star$}
      \begin{tikzcd}
        ss(I^\bullet C^\bullet)^{p,q}_1 \arrow{r}{\partial}\arrow{d}{r_p} &
        ss(I^\bullet C^\bullet)^{p+1,q}_1 \arrow{d}{r_{p+1}}\\
        C^p(G/U, H^q(U,A)) \arrow{r}{(-1)^q\delta} & C^{p+1}(G/U, H^q(U,A))
      \end{tikzcd}
    \end{equation}
    Here $\delta$ denotes the coboundary operator on $C^p$, not on lifted
    maps $U^q\times G^p\sra A.$ By \ref{prop:partial-coboundary-and-rp} \[
      \delta \circ r_p = r_{p+1}\circ\delta_{p+1}.\]

    For $q=0,$ the commutativity follows immediately from the definitions, so
    assume $q\geq 1.$ Take $f\in I^pC^{p+q}$ with $\partial f\in
    I^{p+1}C^{p+q+1}$ Then by \cref{prop:shuffling-and-partial-coboundary} and
    multiple applications of \cref{prop:shuffling-is-identity-if-deep-enough},
    \begin{align*}
      r_{p+1}(\partial f) ={}& r_{p+1}(\shuffle_{p+1}^{p+q+1}(\partial f))\\
      ={}& r_{p+1}\left(  \partial_q (\shuffle_{p+1}^{p+q} f) + (-1)^q \delta_{p+1}(\shuffle_p^{p+q} f) \right)\\
      ={}& r_{p+1}(\partial_q(\shuffle_{p+1}^{p+q}f)) + (-1)^q\delta(r_p(\shuffle_p^{p+q} f)) \\
      ={}&r_{p+1}(\partial_q(\shuffle_{p+1}^{p+q}f)) + (-1)^q\delta(r_p( f)) .
    \end{align*}
    As clearly $r_{p+1}(\partial_q(\shuffle_{p+1}^{p+q} f))=0 $ in $C^{p+1}(G/U, H^q(U,A))$
    by \ref{prop:partial-coboundary-and-rp}, this finishes the proof for open
    subgroups.

    For normal subgroups $N$, the spectral sequence \[ss(I^\bullet
      C^\bullet)^{p,q}_2 \Longrightarrow H^{p+q}(G,A)\] still yields a five term
    exact sequence and we are left with showing that the groups $ss(I^\bullet
    C^\bullet)^{1,0}_2$, $ss(I^\bullet C^\bullet)^{2,0}_2$ and $ss(I^\bullet
    C^\bullet)^{0,1}_2$ are precisely the cohomology groups we were looking for.
    For $q=0,$ the same as above argument works, using
    \ref{prop:iso-on-e1-for-closed-subgrp} instead of
    \ref{prop:hochschild-serre-iso-on-e1}. For the case of $p=0,q=1,$ we cannot
    use \ref{diag:hs-comparing-e1-differentials}. But the same argument as above
    yields a commutative diagram
    \[
      \begin{tikzcd}
        ss(I^\bullet C^\bullet)^{0,1}_1 \arrow{rr}{\partial}\arrow{d}{r_0} & &
        ss(I^\bullet C^\bullet)^{1,1}_1 \arrow[hook]{d}\\
        H^1(N,A) \arrow{r}{=} & C^0(\forget(G/N), H^1(N,A)) \arrow{r}{(-1)\delta} & C^{1}(\forget(G/N), H^1(N,A))
      \end{tikzcd}
    \]
    where the map on the left is induced by the restriction of $f\in C^1(G,A)$
    to $N$. The map on the right, defined analogously to before, is however only
    injective.
    
    The map on the left is
    however still bijective, adapting the proof of
    \ref{prop:hochschild-serre-iso-on-e1}: Represent an element of $ss(I^\bullet
    C^\bullet)^{0,1}_1$ by $f\in C^1(G,A).$ Assume its restriction to $N$ is given
    by $f(n)= n.a-a$ for some $a\in A$. Consider $h = f - (x\mapsto x.a-a)$,
    which is the same as $f$ in \[ss(I^\bullet C^\bullet)^{0,1}_1=\frac{\ker C^1\ra
        C^2/I^1C^2}{\partial(C^0)+I^1C^1}.\] We will show
    that indeed $h\in I^1C^1$ and that hence $h$ and
    therefore $f$ is zero in $ss(I^\bullet C^\bullet)^{0,1}_1$. By assumption,
    $\partial f(x,n)=0$ for all $x\in G,n\in N$, so actually \[ f(xn) = x.f(n) +
      f(x)\] We immediately find that
    \[h(xn) =f(xn) - xn.a + a = x.(n.a-a) + f(x) - xn.a + a = f(x)- (x.a - a) =
      h(x),\] so $h\in I^1C^1$ by \ref{prop:n-invariance-enough-for-i-filt}.

    For surjectivity, choose a representative $\widetilde{f}\in C^1(N,A)$ and
    simply define $f$ via $\widetilde{f}\circ (-)_N.$

    Therefore, \[ss(I^\bullet C^\bullet)^{0,1}_2\cong \ker\delta = \{ f\in
      H^1(N,A)\mid g.f-f=0 \text{ for all } g\in G\},\] which is precisely
    $H^1(N,A)^{\forget(G/N)}$, as $N$ already operates trivially by
    \ref{prop:group-operates-trivially-on-its-cohomology}.
  \end{proof}
\end{theorem}
\begin{remark}
  With all this effort, we still cannot recover the Hochschild-Serre spectral
  sequence for Hausdorff compactly generated topological groups $G$ with closed
  normal subgroup $N$ and discrete coefficients $A$. From the point of view
  presented above, the spectral sequence \[ H^p(G/N,H^q(N,A))\Longrightarrow
    H^{p+q}(G,A)\] is actually an anomaly: The category $\cat{C}$ would be the
  category of compactly generated weakly Hausdorff spaces, which is cartesian
  closed, where the exponential objects are given by $\Hom_{\cat{C}}(X,Y)$,
  endowed with the compact-open topology
  (cf.\,\ref{rmk:internal-hom-in-cgwh-and-discrete-objects}). The analogue of
  \ref{prop:hochschild-serre-iso-on-e1}, which shows an isomorphism of the
  $E_1$-page, is then generally only a bijection -- but for discrete $A$,
  $C^q(N,A)$ (and hence also $H^q(N,A)$) is again discrete and bijectivity then
  suffices for showing the isomorphism.

  In any case, it is much more convenient to derive said spectral sequence from
  homological algebra and reserve the \emph{direct method} for cases where the
  homological arguments fail.
\end{remark}

\section{A Double Complex}
\label{sec:double-complex}
This section is devoted to making a precise statement of the following sort and
proving it afterwards:

\begin{prototheorem}
  Let $G$ be a topologised monoid, $D$ an abelian discrete monoid and $A$ a topologised
  $D\times G$-module. Then in the derived category of abelian groups the
  following holds:
  \[ C^\bullet(D\times G, A) \cong \total C^\bullet(D,C^\bullet(G,A)),\]
\end{prototheorem}

This is very much related to the previous results: There, we filtered the
complex on the left hand side. If we are not looking at a direct product
$D\times G$, there is no double complex on the right hand side -- but a
hypothetical double complex would have $C^\bullet(D,H^\bullet(G,A))$ as the
cohomology in one direction. We compared this cohomology with the $E_1$-page of
the filtered complex and showed that indeed they coincide.

\subsection{Setup and Precise Statement}

For the whole section, we fix:
\begin{itemize}
\item a topological category $\cat{C}$,
\item a topologised monoid $G$ in $\cat{C}$ as in \ref{sec:setup-topological-cochains},
\item a discrete abelian monoid $D$,
\item $-^*\colon D\times G\ra D\times G,$ the morphism of topologised monoids
  which on the level of sets is given by $(d,g)\mapsto (d,1),$
\item $-_G\colon D\times G\ra D\times G,$ the morphism of topologised monoids
  which on the level of sets is given by $(d,g)\mapsto (1,g),$
  \item the canonical projections $\pi_D\colon D\times G\ra D, \pi_G\colon
    D\times G\ra G,$ and
  \item such a $D\times G$-module with $\cat{C}$-rigidification $A,$ that $D$ is
    $A$-pliant.
\end{itemize}

As before, $C^n(G,A)$ denotes the set of normalised (inhomogeneous) cochains
$G^n\sra A.$ We write $C^\bullet$ for $C^\bullet(D\times G,A)$ and denote the
boundary operator of \ref{sec:setup-topological-cochains} by $\partial.$ The
filtration $I^\bullet C^n$ will be taken with respect to the submonoid $G$ of
$D\times G.$

\begin{lemma}\label{prop:direct-product-yields-module-structure-on-cochains}
  The assignment
  \[(\mbof{d}{D}f)(\mbof{\underline{x}}{G^n}) =df(\underline{x})\left(=  df(d^{-1}\underline{x}d)\right) \]
  gives $C^n(G,A)$ the structure of a $D$-module.
  \begin{proof}
    Clear from \ref{prop:action-on-cochains-and-cohomology}.
  \end{proof}
\end{lemma}
\begin{remark}
  This only works because of the direct product structure of $D\times G$,
  cf.~\cref{sec:meaning-of-ci-g-mod-n-cj-n-a}.
\end{remark}

\begin{definition}
  Denote by $C^{\bullet,\bullet}$ the commutative double complex
  $C^p(D,C^q(G,A))$ with differentials
  \begin{align*}
    \delta\colon C^p(D,C^q(G,A))\ra C^{p+1}(D,C^q(G,A))\\
    \partial\colon C^p(D,C^q(G,A))\ra C^p(D,C^{q+1}(G,A))\\
  \end{align*}
  explicitly given by
  \begin{align*}
    \delta(f)(\mbof{\underline{y}}{D^{p+1}})(\mbof{\underline{x}}{G^q}) &= y_1 
                                                                          f(y_2,\dots,y_{p+1})(y_1^{-1}\underline{x}y_1) + (-1)^{p+1}f(y_1,\dots,y_p)(\underline{x})\\
    &\quad\quad + \sum_{i=1}^p f(y_1,\dots,y_{i-1},y_iy_{i+1},y_{i+2},\dots,y_{p+1})(\underline{x})
  \end{align*}
  and
\begin{align*}
    \partial(f)(\mbof{\underline{y}}{D^{p}})(\mbof{\underline{x}}{G^{q+1}}) &= x_1 
                                                                          f(\underline{y})(x_2,\dots,x_{q+1}) + (-1)^{q+1}f(\underline{y})(x_1,\dots,x_q)\\
    &\quad\quad + \sum_{i=1}^q f(\underline{y})(x_1,\dots,x_{i-1},x_ix_{i+1},x_{i+2},\dots,x_{p+1}).
  \end{align*}
  That this is indeed a double complex follows from
  \ref{prop:direct-product-yields-module-structure-on-cochains} and the previous
  discussion in \ref{sec:setup-topological-cochains}. We form the total
  complex \[ \left( \total C^{\bullet,\bullet}\right)^n = \bigoplus_{p+q=n}
    C^{p,q}\] with total differential
 \begin{gather*}
    \Delta\colon C^{p,q}\ra C^{p+1,q} \oplus C^{p,q+1}\\
   \Delta = \partial + (-1)^q\delta
 \end{gather*}

\end{definition}

\begin{remark}\label{rmk:total-complex-single-operator}
  If $D\cong \N_0$ (or $D\cong\Z$) operates via a single operator $\varphi,$
  then because of \ref{prop:free-resolution-of-z-is-resolution} and the fact
  that \[ 0 \ra \Z[\varphi]\oset{\varphi -1}{\ra} \Z[\varphi] \ra \Z\ra 0\] is
  also a free resolution of the integers, we see that (by abstract nonsense) \[
    \total C^{\bullet,\bullet} \cong \total \left( C^\bullet(G,A)\oset{\varphi
        -1}{\ra} C^{\bullet}(G,A) \right)\] in the derived category of abelian
  groups. This immediately generalises to monoids $D\cong \N_0^r\times\Z^s$ by
  induction.
\end{remark}

Our main result of this section can now be stated as follows:

\begin{theorem}\label{prop:quasi-iso-in-direct-product-case}
  There is a quasi-isomorphism of complexes \[ C^\bullet\ra \total
    C^{\bullet,\bullet}.\]
\end{theorem}

The preparations of its proof will span the next couple of pages, which itself
is given on page~\pageref{proof:quasi-iso-in-direct-product-case}.

\subsection{The morphism ...}

\begin{definition}
  For $f\in C^{p+q}$ denote by $r_p(f)\in C^{p,q}$ the map \[
    r_p(f)(\mbof{\underline{y}}{D^p})(\mbof{\underline{x}}{G^q}) =
    f((1,x_1),\dots,(1,x_q), (y_1,1),\dots, (y_p,1))
  \]
  and by $\alpha$ the map \begin{gather*}\alpha\colon C^n\ra \bigoplus_{p+q=n} C^{p,q}\\\alpha(f) =
    \bigoplus_{p+q=n} r_p(\shuffle_p^{p+q}(f)).\end{gather*}
\end{definition}
\begin{proposition}
 $\alpha$ is a morphism of complexes $C^\bullet \ra \total C^{\bullet,\bullet}$.
 \begin{proof}
   Consider the diagram
   \[
     \begin{tikzcd}
       C^{n-1} \arrow{rr}\arrow{d}{\partial}& &  \bigoplus_{p+q=n-1} C^{p,q}\arrow{d}{\Delta}\\
       C^{n}\arrow{rr} &&\bigoplus_{p+q=n-1} \left( C^{p+1,q}\oplus C^{p,q+1} \right).
     \end{tikzcd}
   \]
   To show that it commutes, let $p'+q'=n$. We will compare \[
     r_{p'}(\shuffle_{p'}^n(\partial f))\] with the entry of $\Delta(\alpha(f))$
   in $C^{p',q'}$. By definition, this entry is equal to
   \[ \partial(r_{p'}(\shuffle_{p'}^{p'+q'-1}(f))) +
     (-1)^{q'}\delta(r_{p'-1}(\shuffle_{p'-1}^{p'-1+q'}(f))).\]
   By \ref{prop:partial-coboundary-and-rp}, \[ \partial \circ r_{p'} = r_{p'}\circ\partial_{q'}\]
   and \[\delta\circ r_{p'-1} = r_{p'}\circ\delta_{p'},\] where $\partial_{q'}$
   and $\delta_{p'}$ are the maps from \ref{def:partial-coboundary-operators}.
   The claim then follows immediately from the additivity of $r_{p'}$ and
   \ref{prop:shuffling-and-partial-coboundary}.
 \end{proof}
\end{proposition}

\subsection{\dots\ and its quasi-inverse}

\begin{lemma}
  \label{prop:sharp-basic-properties}
  Consider the map
  \begin{gather*}
    (-)^\sharp \colon C^{p,q}\ra C^{p+q},\\
    f^\sharp (\mbof{\underline{z}}{(D\times G)^{p+q}}) = f(\pi_D(z_{q+1}),\dots,\pi_D(z_{p+q}))(\pi_G(z_1),\dots,\pi_G(z_q)).
  \end{gather*}
  Its image lies in $I^pC^{p+q}$ and the
  composition \[ \begin{tikzcd}C^{p,q}\arrow{r}{-^\sharp} & I^p
      C^{p+q}\arrow{r}{r_p} & C^{p,q}\end{tikzcd}\] is the identity.
  \begin{proof}
    Clear from the definitions.
  \end{proof}
\end{lemma}
\begin{proposition}
  \label{prop:section-of-comparison-as-module}
  $(-)^\sharp$ induces a map \[ (-)^\sharp\colon (\total C^{\bullet,\bullet})^n\ra C^n \]
  with \[ \alpha\circ(-)^\sharp = \id_{(\total C^{\bullet,\bullet})^n}.\]
  \begin{proof}
    Let $f\in C^{p,q}$.
    \Cref{prop:sharp-basic-properties,prop:shuffling-is-identity-if-deep-enough}
    (with $N=G$) imply that \[ r_p(\shuffle_p^{p+q}(f^\sharp)) = f.\] It
    now remains to show that for $p'\neq p$,
    $r_{p'}(\shuffle_{p'}^{p+q}(f^\sharp))=0.$

    Let $\varphi\colon\ceil{p'}\ra\ceil{p+q}$ be an injective map of
    ordered sets, so
    \begin{align*} 
      r_{p'}((f^\sharp)^\phi)(\mbof{\underline{y}}{D^{p'}})(\mbof{\underline{x}}{G^{p+q-p'}})
      &=(f^\sharp)^\phi\left( (1,x_1),\dots, (1,x_{p+q-p'}), (y_1,1),\dots, (y_{p'},1) \right)\\
      &=f^\sharp(\gamma_1,\dots,\gamma_{p+q}),
    \end{align*}
    with $\gamma_k$ a conjugate of one of the $(1,x_i)$ or one of the $(y_i,1)$.
    Therefore, at least $p'$ of the $\gamma_k$ have $\pi_G(\gamma_k)=1$ and at
    least $p+q-p'$ of the $\gamma_k$ have $\pi_D(\gamma_k)=1.$
    Now \[f^\sharp(\gamma_1,\dots,\gamma_{p+q}) =
      f(\pi_D(\gamma_{q+1}),\dots,\pi_D(\gamma_{p+q}))(\pi_G(\gamma_1),\dots,\pi_G(\gamma_q)),\]
    and all cocycles are normalised, so this can only be non-zero if all
    $\gamma_k$ with $\pi_D(\gamma_k)=1$ are among the first $q$, so \[p+q-p'\leq
      q\] and if all $\gamma_k$ with $\pi_G(\gamma_k)=1$ are among the last $p$,
    so \[ p' \leq p.\] But this is impossible if $p'\neq p.$ Therefore
    $r_{p'}((f^\sharp)^\phi)=0$ and hence
    also \[r_{p'}(\shuffle_{p'}^{p+q}(f^\sharp))=0.\]
  \end{proof}
\end{proposition}
\begin{remark}
  The map $(-)^\sharp$ of \ref{prop:section-of-comparison-as-module} is not a
  map of complexes, so while it is easy to construct preimages in the direct
  product case, these are not particularly useful. Showing that $\alpha$ is a
  quasi-isomorphism hence again uses the calculations of \ref{sec:extension-of-cochains}.
\end{remark}

\begin{proposition}
  \label{prop:explicit-determination-of-ext-by-zero-for-products}
  Let $u\in C^{p,q}$ and $g=g(u,0)\in I^pC^{p+q}$ its extension along $0$ from
  \ref{prop:cochain-ext-main-work}. Then \[ g(\mbof{x_1}{D\times
      G},\dots,\mbof{x_q}{D\times G},\mbof{\underline{y}}{(D\times G)^p}) =
    x_1^*\dots x_q^* . u^\sharp(x_1,\dots, x_q,\underline{y}).\]
  \begin{proof}
    Note first that by definition of $-^\sharp,$
    \[ u^\sharp(z_1,\dots,z_q, z_1',\dots,z_p') =
      u^\sharp((z_1)_G,\dots,(z_q)_G, z_1'^*,\dots,z_p'^*).\] Define as in
    \ref{prop:cochain-ext-main-work} \[ g_1(\mbof{x_1}{D\times G},
      \mbof{\underline{\sigma}}{G^{q-2}}, \mbof{\underline{y}}{(D\times G)^p}) =
      x_1^*. u^\sharp((x_1)_G,\underline{\sigma},\underline{y}).
    \] and $g_k = \ext_0(g_{k-1})$ for $2\leq k \leq q,$ so that $g=g_q.$ We
    will inductively show that \[ g_k(\mbof{x_1}{D\times G},\dots,
      \mbof{x_k}{D\times G},
      \mbof{\underline{\sigma}}{G^{q-k}},\mbof{\underline{y}}{(D\times G)^p}) =
      x_1^*\dots x_k^* .
      u^\sharp(x_1,\dots,x_k,\underline{\sigma},\underline{y}),\] which
    is trivial for $k=1$.
    By definition of the extension,
      \[g_{k+1}(\mbof{x_1}{D\times G}, \dots, \mbof{x_k}{D\times G},
      \mbof{x_{k+1}}{D\times G},
      \mbof{\underline{\sigma}}{G^{q-k-1}},\mbof{\underline{y}}{(D\times G)^p})
      = g_k(x_1,\dots, x_{k-1}, x_kx_{k+1}^*, (x_{k+1})_G, \underline{\sigma},
      \underline{y}), \]
    which by induction hypothesis is exactly
    \[x_1^* \dots x_{k-1}^* \cdot (x_k x_{k+1}^*)^*. u^\sharp((x_1)_G,\dots,
      (x_{k-1})_G, (x_kx_{k+1}^*)_G, (x_{k+1})_G,
      \underline{\sigma},\underline{y})\] As in our case $-^*$ and $-_G$ are
    homomorphisms with $-_G\circ -^*\equiv 1$, this shows the proposition.
  \end{proof}
\end{proposition}

\begin{corollary}\label{prop:cochain-lift-trivial-if-first-arguments-in-quotient}
  Let $u\in C^{p,q}$ and define $g=g(u,0)$ as in \ref{prop:cochain-ext-main-work}.
  Then
  \[ g(\mbof{\underline{x}}{(D\times G)^{q}},\mbof{\underline{y}}{(D\times
      G)^p}) = 0 \] if one of the first $q$ arguments lies in $D$.
  \begin{proof}
    Clear from \ref{prop:explicit-determination-of-ext-by-zero-for-products} and
    the definition of $-^\sharp$.
  \end{proof}
\end{corollary}
\begin{proposition}\label{prop:ext-by-zero-is-section-of-alpha}
  Let $u\in C^{p,q}$ and $g=g(u,0)\in I^pC^{p+q}$ its extension along $0$ from
  \ref{prop:cochain-ext-main-work}. Then
  $\alpha(g)=(0,\dots,0,\mbof{u}{C^{p,q}},0,\dots,0).$
  \begin{proof}
    As $g\in I^pC^{p+q},$ we have for all $p'<p$ \[ \alpha(g)^{p',
        p+q-p'}=r_{p'}(\shuffle_{p'}^{p+q}(g)) = r_{p'}(g)=0 \] by
    \ref{prop:shuffling-is-identity-if-deep-enough,prop:rp-trivial-if-too-deep}.
    By \ref{prop:cochain-ext-main-work,prop:shuffling-is-identity-if-deep-enough}, $\alpha(g)^{p,q}=u,$ so it
    remains to show that $\alpha(g)^{p', p+q-p'}=0$ for $p'>p,$ i.\,e., that \[
      \shuffle_{p'}^{p+q} g (\mbof{\underline{d}}{D^{p'}},
      \mbof{\underline{y}}{G^{p+q-p'}}) = 0.\] But the definition of the shuffle
    operator implies that this is the sum of values of the form \[ \pm
      g(\mbof{\underline{\gamma}}{(D\times G)^{p+q}})\] where at least $p'$
    arguments lie in $D.$ As $p'>p$, one of these arguments that lie in $D$ is
    in one of the first $q$ positions, so $g(\underline{\gamma})=0$ by
    \ref{prop:cochain-lift-trivial-if-first-arguments-in-quotient}.
  \end{proof}
\end{proposition}
\begin{proposition}\label{prop:ext-by-zero-is-morphism}
  Extension along zero is a morphism of complexes $\total C^{\bullet,\bullet}\ra
  C^\bullet.$
  \begin{proof}
    We need to show the following: Let $u\in C^{p,q}$ with $\Delta u= v + w$
    with $v\in C^{p+1,q}$ and $w\in C^{p,q+1}$. Call their respective extensions
    along zero from \ref{prop:cochain-ext-main-work} \begin{align*}g&=g(u,0)\in I^pC^{p+q},\\
      h&=g(v,0)\in I^{p+1}C^{p+q+1}, \text{ and }\\ h'&=g(w, 0) \in I^pC^{p+q+1}.\end{align*}
    Then \[ \partial g = h+h'.\]

    Using \ref{prop:explicit-determination-of-ext-by-zero-for-products}, this is
    now a straight forward (albeit lengthy) calculation.
    First of all,
    \begin{align*}
      \partial g(\mbof{\underline{x}}{(D\times G)^{p+q+1}}) &= x_1 . g(x_2,\dots,x_{p+q+1}) + (-1)^{p+q+1} g(x_1,\dots, x_{p+q}) \\
                                                          &\quad\quad + \sum_{i=1}^{p+q} (-1)^i g(x_1,\dots, x_{i-1}, x_ix_{i+1}, x_{i+2},\dots, x_{p+q+1})\\
                                                          & = x_1\cdot x_2^*\dots x_{q+1}^* . u^\sharp(x_2,\dots, x_{q+1}, x_{q+2},\dots,x_{p+q+1})\\
                                                          &\quad\quad + (-1)^{p+q+1}(x_1\dots x_q)^* . u^\sharp (x_1,\dots, x_q, x_{q+1},\dots x_{p+q})\\
     \tag{$\Sigma$.1}                                                     &\quad\quad + \sum_{i=1}^{q} (-1)^i (x_1\dots x_{q+1})^* u^\sharp(x_1,\dots, x_i x_{i+1}, \dots, x_{p+q+1}) \\
    \tag{$\Sigma$.2}                                                      &\quad\quad + \sum_{i=q+1}^{p+q}(-1)^i (x_1\dots x_q)^* u^\sharp(x_1,\dots, x_{i-1}, x_i x_{i+1}, x_{i+2},\dots, x_{p+q+1} ) 
    \end{align*}
    Expanding $h$ we first get
    \begin{align*}
      h(\mbof{\underline{x}}{(D\times G)^{p+q+1}}) &= (x_1\dots x_q)^* v^\sharp(x_1,\dots, x_q, x_{q+1},\dots, x_{p+q+1}) \\
                                                   &= (x_1\dots x_q)^* v(x_{q+1}^*,\dots, x_{p+q+1}^*)((x_1)_G,\dots, (x_q)_G).
    \end{align*}
    We can furthermore express $(-1)^q v(x_{q+1}^*,\dots, x_{p+q+1}^*)((x_1)_G,\dots,
    (x_q)_G)$ as follows:
    \begin{align*}
      &(-1)^q v(x_{q+1}^*,\dots, x_{p+q+1}^*)((x_1)_G,\dots,
    (x_q)_G)\\
      &= x_{q+1}^* . u(x_{q+2}^*,\dots, x_{p+q+1}^*)((x_1)_G,\dots, (x_q)_G)\\
      &\quad\quad + (-1)^{p+1} u(x_{q+1}^*,\dots,x_{p+q}^*)((x_1)_G,\dots,(x_q)_G) \\
      &\quad\quad + \sum_{i=1}^p (-1)^i u(x_{q+1}^*,\dots, x_{q+i-1}^*, (x_{q+1}x_{q+i+1})^*, x_{q+i+2}^*,\dots, x_{p+q+1}^*)((x_1)_G,\dots, (x_q)_G)\\
      &= x_{q+1}^* . u^\sharp(x_1,\dots, x_q,x_{q+2},\dots, x_{p+q+1})\\
      &\quad\quad + (-1)^{p+1} u^\sharp(x_1,\dots,x_q,x_{q+1},\dots,x_{p+q})\\
      &\quad\quad + \sum_{i=1}^p (-1)^i u^\sharp(x_1,\dots, x_{q+i-1}, x_{q+i}x_{q+i+1}, x_{q+i+2},\dots, x_{p+q+1}).
    \end{align*}
    On the other hand,
    \begin{align*}
      h'(\mbof{\underline{x}}{(D\times G)^{p+q+1}}) &= (x_1\dots x_{q+1})^* w^\sharp(x_1,\dots, x_{q+1}, x_{q+2},\dots,x_{p+q+1})\\
                                                    &= (x_1\dots x_{q+1})^* w(x_{q+2}^*,\dots, x_{p+q+1}^*)((x_1)_G,\dots, (x_{q+1})_G),
    \end{align*}
    and we can express $w(x_{q+2}^*,\dots, x_{p+q+1}^*)((x_1)_G,\dots,
    (x_{q+1})_G)$ as follows:
    \begin{align*}
      & w(x_{q+2}^*,\dots, x_{p+q+1}^*)((x_1)_G,\dots,
      (x_{q+1})_G)\\
      &= (x_1)_G . u(x_{q+2}^*,\dots, x_{p+q+1}^*)((x_2)_G,\dots, (x_{q+1})_G)\\
      &\quad\quad + (-1)^{q+1} u(x_{q+2}^*,\dots, x_{p+q+1}^*)((x_1)_G,\dots, (x_{q})_G)\\
      &\quad\quad + \sum_{i=1}^q (-1)^{i} u(x_{q+2}^*,\dots, x_{p+q+1}^*)((x_1)_G,\dots, (x_{i-1})_G, (x_ix_{i+1})_G, x_{i+2},\dots, (x_{q+1})_G)\\
    &= (x_1)_G . u^\sharp(x_2,\dots, x_{q+1},x_{q+2},\dots, x_{p+q+1})\\
      &\quad\quad + (-1)^{q+1} u^\sharp(x_1,\dots, x_{q},x_{q+2},\dots, x_{p+q+1})\\
      &\quad\quad + \sum_{i=1}^q (-1)^{i} u^\sharp(x_1,\dots, x_{i-1}, x_ix_{i+1}, x_{i+2},\dots, x_{q+1},x_{q+2},\dots, x_{p+q+1}).
    \end{align*}
    We see at once that ($\Sigma$.1) appears in our expansion of $h'$ and that
    ($\Sigma$.2) appears in our expansion of $h$. The remaining terms are as follows:
    \begin{align*}
      & (\partial g - h -h')(\mbof{\underline{x}}{(D\times G)^{p+q+1}}) \\
      & \tag{$\star$.1} = x_1\cdot x_2^*\dots x_{q+1}^* . u^\sharp(x_2,\dots, x_{p+q+1})\\
      \tag{$\star$.2}                                                              &\quad\quad + (-1)^{p+q+1}(x_1\dots x_q)^* . u^\sharp (x_1,\dots,x_{p+q})\\
      \tag{$\star$.3}                                                             &\quad\quad - (-1)^q (x_1\dots x_q)^* x_{q+1}^* u^\sharp(x_1,\dots,x_q, x_{q+2},\dots, x_{p+q+1})\\
      \tag{$\star$.4}                                              &\quad\quad - (-1)^{p+q+1} (x_1\dots x_q)^* u^\sharp(x_1,\dots,x_{p+q}) \\
      \tag{$\star$.5}                               &\quad\quad - (x_1\dots x_{{q+1}})^* (x_1)_G u^\sharp(x_2,\dots,x_{p+q+1}) \\
      \tag{$\star$.6}                &\quad\quad - (-1)^{q+1}(x_1\dots x_{{q+1}})^* u^\sharp(x_1,\dots,x_q,x_{q+2},\dots,x_{p+q+1}) .
    \end{align*}
    By construction, \[ x_1\cdot x_2^*\dots x_{q+1}^* = (x_1)_G\cdot x_1^*\dots
      x_{q+1}^* = (x_1\dots x_{{q+1}})^* (x_1)_G, \] so ($\star$.1) and
    ($\star$.5) cancel. Also ($\star$.2) and ($\star$.4) cancel, as do
    ($\star$.3) and ($\star$.6).
  \end{proof}
\end{proposition}

Armed with this, we are now ready to prove the main result of this section.

\begin{proof}[Proof of \ref{prop:quasi-iso-in-direct-product-case}]
  \label{proof:quasi-iso-in-direct-product-case}
  Indeed $\alpha$ is the required quasi-isomorphism
  \[\alpha\colon C^\bullet\oset{\approx}{\ra}\total C^{\bullet,\bullet}.\]
  Surjectivity on the level of cohomology follows immediately from
  \ref{prop:ext-by-zero-is-morphism,prop:ext-by-zero-is-section-of-alpha}.

  It remains to see that $\alpha$ is injective on cohomology. For this matter,
  take $f\in C^{n}$ with $\partial f=0$ and $\alpha(f)=\Delta(u)$ for some $u\in
  (\total C^{\bullet,\bullet})^{n-1}$. Write $\alpha(f) = (f^{p,q})_{p,q}\in
  \bigoplus_{p+q=n} C^{p,q}$.  We will now modify $f$ step
  by step by elements of $\partial(C^{n-1})$ such that it lies in higher and
  higher $I^pC^n$ until it lies in $I^{n+1}C^n=0$, i.\,e., $f$ is
  cohomologous to zero.

  Let $\widetilde{f}\in C^n$ and $\widetilde{u}\in (\total
  C^{\bullet,\bullet})^{n-1}$. We call the tuple $(\widetilde{f},\widetilde{u})$
  better than $(f,u)$ at $p$ if the following hold:
  \begin{enumerate}
  \item $f-\widetilde{f}\in \partial(C^{n-1})$,
  \item $\widetilde{f}\in I^pC^n$,
  \item $\alpha(\widetilde{f})=\Delta(\widetilde{u})$, and
  \item $\widetilde{u}^{k,n-1-k}=0$ for $k<p$.
  \end{enumerate}

  We will inductively construct an $\widetilde{f}\in C^n,$ such that
  $(\widetilde{f},0)$ is better than $(f,u)$ at $n$. We will afterwards show
  that this $\widetilde{f}$ is already zero and hence $f\in \partial(C^{n-1})$.
  
  Obviously $(f,u)$ itself is better than $(f,u)$ at $0$. If
  $(\overline{f},\overline{u})$ is better than $(f,u)$ at $p$, we construct a
  tuple $(\widetilde{f},\widetilde{u})$ which is better than $(f,u)$ at $p+1$ as
  follows:
  Note that analogously to \ref{prop:ext-by-zero-is-section-of-alpha}, by
  \ref{prop:shuffling-is-identity-if-deep-enough,prop:rp-trivial-if-too-deep}, 
  $\overline{f}^{p',q}=0$ for all $p'<p$ and
  $\overline{f}^{p,n-p}=r_p(\overline{f})$.
  By assumption, \[r_p(\overline{f})=\partial \overline{u}^{p,n-1-p}.\]
  If $p\leq n-2$, we can do the following: By
  \ref{prop:cochain-ext-main-work} (with $u=\overline{u}^{p,n-p-1},
  f=\overline{f}, p=p, q=n-p$) we find $g\in I^pC^{n-1}$ with the following
  properties:
  \begin{enumerate}
  \item $\alpha(g)^{p',n-1-p'} = \overline{u}^{p',n-1-p'}$ for all $p'\leq p$,
  \item $\overline{f}-\partial(g)\in I^{p+1}C^n$.
  \end{enumerate}
  Note that for \ref{prop:cochain-ext-main-work} to be applicable, we need the
  assumption that $p\leq n-2.$
  
  Now set $\widetilde{f}=\overline{f}-\partial(g)$ and
  $\widetilde{u}=\overline{u}-\alpha(g).$ To show that $(\widetilde{f},\widetilde{u})$ is
  better than $(f,u)$ at $p+1$, we only have to show
  that \[\alpha(\widetilde{f})=\Delta(\widetilde{u}),\] but this is straight
  forward:

  \[ \alpha(\widetilde{f}) = \alpha(\overline f) - \alpha(\partial g) =
    \Delta(\overline u) - \Delta(\alpha(g)) = \Delta(\widetilde{u}).\]

  Repeating this process, we get a tuple $(\overline{f},\overline{u})$, which is
  better than $(f,u)$ at $n-1$, so \[ \overline{u} = 0 \oplus 0 \oplus \dots
    \oplus 0 \oplus \overline{u}^{n-1,0} \] and \[ \alpha(\overline{f}) = 0
    \oplus \dots \oplus 0 \oplus \overline{f}^{n-1,1} \oplus
    \overline{f}^{n,0}.\] Now set $g=(\overline{u}^{n-1,0})^\sharp \in
  I^{n-1}C^{n-1}(D\times G,A), \widetilde{f} = \overline{f}-\partial(g),
  \widetilde{u}=\overline{u}-\alpha(g).$ (Note that by construction,
  $\widetilde{u}=0.$)

  It is immediately clear that \[ \alpha(\widetilde{f}) = \Delta(\widetilde{u})
    = 0.\] To see that $(\widetilde{f},0)$ is better than $(f,u)$ at $n$, it
  only remains to show that $\widetilde{f}\in I^nC^n.$ As it is clear from the
  construction that $\widetilde{f}\in I^{n-1}C^n,$ we only need to show
  that \[\widetilde{f}\left(
      (\mbof{d}{D},\mbof{\sigma}{G}),\mbof{\underline{x}}{(D\times
        G)^{n-1}}\right) = \widetilde{f}((d,1), \underline{x}).\]

  The equality $\overline{f}^{n-1,1} = \partial \overline{u}^{n-1,0}$ (together
  with $\overline{f}\in I^{n-1}C^n$) implies
  \begin{align*}\overline{f}((1,\mbof{\sigma}{G}),
    \underline{x}) &= r_{n-1}(\overline{f})(\pi_D(\underline{x}))(\sigma)\\
                   &= \overline{f}^{n-1,1}(\pi_D(\underline{x}))(\sigma)\\
                   &= \partial(\overline{u}^{n-1,0}(\pi_D(\underline{x})))(\sigma)\\
                   &= \sigma
                     \overline{u}^{n-1,0}(\pi_D(\underline{x})) - \overline{u}^{n-1,0}(\pi_D(\underline{x}))\\
                   &= (1,\sigma) (\overline{u}^{n-1,0})^\sharp(\underline{x}) - (\overline{u}^{n-1,0})^\sharp(\underline{x}).
  \end{align*}

  As $\partial(\overline{f})=0$ and $\overline{f}\in I^{n-1}C^n,$ the
  coboundary expansion of $\partial\overline{f}((d,1),(1,\sigma),\underline{x})=0$
  yields
  \begin{align*}
    \overline{f}((d,\sigma),\underline{x}) &=
    (d,1)\overline{f}((1,\sigma),\underline{x}) +
    \overline{f}((d,1),\underline{x})\\ & =
                                          (d,1)\overline{f}((1,\sigma),\underline{x}) +\text{ terms independent of } \sigma
  \end{align*}
  and analogously \[
    \partial(g)((d,\sigma),\underline{x}) = (d,\sigma)g(\underline{x}) + \text{
      terms independent of } \sigma.\]
  We can hence compute:
  \begin{align*}
    \widetilde{f}((d,\sigma), \underline{x}) &= (\overline{f}-\partial(g))((d,\sigma), \underline{x})\\
    & = (d,1)\overline{f}((1,\sigma),\underline{x}) - (d,\sigma)g(\underline{x}) + \text{ terms independent of } \sigma\\
                                             &= (d,\sigma)(\overline{u}^{n-1,0})^\sharp(\underline{x}) - (d,1)(\overline{u}^{n-1,0})^\sharp(\underline{x}) - (d,\sigma)(\overline{u}^{n-1,0})^\sharp(\underline{x})\\
    & \quad\quad + \text{ terms independent of } \sigma,
  \end{align*}
  which is independent of $\sigma$, hence $\widetilde{f}\in I^nC^n.$

  We conclude the proof by showing that if $\widetilde{f}\in
  I^nC^n\cap\ker\alpha\cap\ker\partial,$ then $\widetilde{f}\in\partial
  C^{n-1}.$ But indeed such an $\widetilde{f}\in I^nC^n\cap\ker\alpha$ is already zero:
  \begin{align*}
    \widetilde{f}((\mbof{d_1}{D},\mbof{x_1}{G}), \dots, (\mbof{d_n}{D},\mbof{x_n}{G})) &= \widetilde{f}((d_1,1),\dots, (d_n,1))\\
                                                                                       &= r_n(\widetilde{f})(d_1,\dots,d_n)\\
& = \alpha(\widetilde{f})^{n,0}(d_1,\dots,d_n) = 0.
  \end{align*}
\end{proof}

\subsection{On a theorem of Jannsen}

The main result of \cite{MR1060372} also has a variant in the topological
setting.

We first recall the following result:

\begin{proposition}[{\cite[(2.3.4)]{MR2392026}}]
  \label{prop:total-complex-tensor-product-splitting}
  Let $C^\bullet, D^\bullet$ be complexes of modules over a Dedekind domain $R$.
  Assume that both complexes are bounded in the same direction or that one of
  them is bounded above and below. If $C^\bullet$ consists of flat $R$-modules,
  then there is a non-canonical splitting \[ H^n(\total C^\bullet\otimes_R
    D^\bullet) \cong \bigoplus_{p+q=r} ss(C^\bullet\otimes_R
    D^\bullet)_2^{p,q},\] where $ss(C^\bullet\otimes_R D^\bullet)_2^{p,q}$
  denotes the $E_2$-terms of the spectral sequence attached to the double
  complex (cf.~e.\,g.~\cite[(2.2.3)]{MR2392026} for details).
\end{proposition}

\begin{proposition}\label{prop:cohomology-of-double-complex-with-trivial-outer-action}
  If $\forget D$ is finite and acts trivially on $A$, then \[ H^n(D\times
    G,A) \cong \bigoplus_{p+q=n} H^p(D,H^q(G,A)).\]
  \begin{proof}
    By \ref{prop:quasi-iso-in-direct-product-case} it suffices to show
    that \[C^{\bullet,\bullet}\cong C^{\bullet}(D,\Z)\otimes C^\bullet(G,A)\] as
    double complexes, as we can then employ
    \ref{prop:total-complex-tensor-product-splitting} to get the desired result.
    As $D$ is
    finite, it is clear that
    \begin{align*}
       C^p(D,\Z)\otimes C^q(G,A) &\ra C^p(D,C^q(G,A)) \\
       (f, g) &\mapsto \left(
         \mbof{\underline{d}}{D^p} \mapsto \left( \mbof{\underline{x}}{G^q}
           \mapsto f(\underline{d})g(\underline{x}) \right) \right)
    \end{align*}
    is bijective and it is easily verified that it commutes with differentials.
  \end{proof}
\end{proposition}

The assumption of finite $D$ is regrettably crucial in the proof. In
\cite{MR1060372} the case of \emph{compact} (but not necessarily discrete) $D$
and discrete $A$ is considered. Every morphism $D\sra A$ then has finite image,
which induces the isomorphism above.

However, for
the easiest monoids we also have the following:

\begin{proposition}
  If $D\cong \N_0^r$ (or $D\cong\Z^r$) acts trivially on $A$, then \[
    H^n(D\times G,A) \cong \bigoplus_{k=0}^r H^{n-k}(G,A)^{\oplus
      \binom{r}{k}}.\]
  \begin{proof}
    It suffices to show the proposition for $r=1,$ as the general case then
    follows by induction. By
    \ref{rmk:total-complex-single-operator,prop:quasi-iso-in-direct-product-case} \[
      C^\bullet(D\times G, A) \cong \total\left( C^\bullet(G,A)\oset{0}{\ra}
        C^{\bullet}(G,A) \right) \cong C^\bullet(G,A)\oplus
      C^{\bullet-1}(G,A),\]
    so \[ H^n(D\times G,A) \cong H^n(G,A)\oplus H^{n-1}(G,A).\]
  \end{proof}
\end{proposition}

\section{Shapiro's Lemma for Topologised Monoids}
\label{sec:shapiro-for-monoids}

The results of the previous section allow us to extend Shapiro's lemma to
monoids.

\begin{theorem}
  \label{prop:shapiro-top-mon}
  Let $\cat{C}$ be a topological category, $G$ a topologised group in $\cat{C}$
  and $D$ a discrete monoid. Let $H\leq G$ be a subgroup as in
  \ref{sec:shapiro-top-grps} and $A$ a rigidified $D\times H$-module with $D$
  being $A$-pliant. Then \[ C^\bullet(D\times G, \Ind^H_G(A)) \cong
    C^\bullet(D\times H, A)\] in the derived category of abelian groups.
  \begin{proof}
    Let us first note that $D$ is also
    $\Ind^H_G(A)$-pliant: We need to show that for every $X\in\cat{C}$ we have
    an equality \[ \Ind^H_G(A)(D\times X) = \Hom_{\cat{Set}}(\forget
      D,\Ind^H_G(A)(X)).\] As $D$ is $A$-pliant, $\Ind^H_G(A)(D\times X)$ are
    those maps in $\Hom_{\cat{Set}}(\forget D, h_A(X\times G))$ which are $H$-linear
    in the $G$-argument. But that is exactly $\Hom_{\cat{Set}}(\forget
    D,\Ind^H_G(A)(X))$. 

    We can hence use
    \ref{prop:quasi-iso-in-direct-product-case} to see that
    \[C^\bullet(D\times G, \Ind^H_G(A)) \cong \total
      C^\bullet(D,C^\bullet(G,\Ind^H_G(A))).\] By \ref{prop:c-quasiiso-x}
    we have a quasi-isomorphism \[\total
      C^\bullet(D,C^\bullet(G,\Ind^H_G(A)))\cong\total
      X^\bullet(D,C^\bullet(G,\Ind^H_G(A))).\] As $D$ is discrete,
    $X^\bullet(D,-)=\Hom_{\Z[D]}(F_\bullet, -)$, where $F_\bullet$ is a complex
    of free $\Z[D]$-modules, cf.~\ref{prop:free-resolution-of-z-is-resolution}.
    Thus $X^\bullet(D,-)$
    preserves quasi-iso\-morphisms. Using these arguments again, together
    with \ref{prop:shapiro-top-grps}, we arrive at quasi-isomorphisms 
   \[\total
      X^\bullet(D,C^\bullet(G,\Ind^H_G(A)))\cong\total
      X^\bullet(D,C^\bullet(H,A)) \cong \total C^\bullet(D,C^\bullet(H,A))\cong
      C^\bullet(D\times H,A).\] 
  \end{proof}
\end{theorem}

\printbibliography
\vfill
{\hfill\footnotesize The \textsc{Bib}\TeX-entries for this bibliography were
  mostly taken from MathSciNet.}

\end{document}